\documentclass[12pt,reqno,tbtags]{amsart}

\usepackage{amsthm}
\usepackage{amssymb}
\usepackage{amsfonts, mathrsfs}
\usepackage{amsmath}
\usepackage[numbers, sort&compress]{natbib}
\usepackage{graphicx}
\usepackage{vmargin, enumerate}
\usepackage{setspace}
\usepackage{paralist}
\usepackage[pdftex]{hyperref}
\usepackage[normalem]{ulem}

\newcommand{\R}{\mathbb{R}}

\newcommand{\N}{{\mathbb N}}

\newcommand{\E}[1]{{\mathbf E}\left[#1\right]}

\newcommand{\p}[1]{{\mathbf P}\left\{#1\right\}}

\newcommand{\I}[1]{{\mathbf 1}_{[#1]}}

\newcommand{\ip}[2]{\left\langle #1 , #2 \right\rangle}
\newcommand{\ipo}[1]{\left\langle #1 , #1 \right\rangle}

 \newcommand{\bag}{\begin{align}}
\newcommand{\bags}{\begin{align*}}
\newcommand{\eag}{\end{align*}}
\newcommand{\eags}{\end{align*}}

\newtheorem{thm}{Theorem}
\newtheorem{lem}[thm]{Lemma}
\newtheorem{prop}[thm]{Proposition}
\newtheorem{cor}[thm]{Corollary}

\newtheorem{fact}[thm]{Fact}

\newcommand\cE{\mathcal E}

\newcommand\by{{\bf y}}

\newcommand{\refT}[1]{Theorem~\ref{#1}}
\newcommand{\refC}[1]{Corollary~\ref{#1}}
\newcommand{\refL}[1]{Lemma~\ref{#1}}

\newcommand{\refS}[1]{Section~\ref{#1}}
\newcommand{\refP}[1]{Proposition~\ref{#1}}

\newcommand{\refF}[1]{Fact~\ref{#1}}
\newcommand{\refQ}[1]{(\ref{#1})}

\newcommand{\refeq}[1]{(\ref{#1})}

\newcommand{\pran}[1]{\left(#1\right)}

\newcommand{\eps}{\epsilon}

\hypersetup{
    bookmarks=true,         
    unicode=false,          
    pdftoolbar=true,        
    pdfmenubar=true,        
    pdffitwindow=true,      
    pdftitle={My title},    
    pdfauthor={Author},     
    pdfsubject={Subject},   
    pdfnewwindow=true,      
    pdfkeywords={keywords}, 
    colorlinks=true,       
    linkcolor=blue,          
    citecolor=blue,        
    filecolor=blue,      
    urlcolor=blue           
}

\newcommand\urladdrx[1]{{\urladdr{\def~{{\tiny$\sim$}}#1}}}

\begingroup
  \divide\count255 by 60
  \multiply\count255 by -60
  \advance\count255 by \time
  \ifnum \count255 < 10 \xdef\oclock{\the\count1:0\the\count255}
  \else\xdef\oclock{\the\count1:\the\count255}\fi
\endgroup

\newcommand{\bM}{\mathbf{M}}
\newcommand{\dog}{D^{>0}}
\newcommand{\simg}{\sim_{\Gamma}}
\newcommand{\simh}{\sim_{H}}
\newtheorem{question}[thm]{Question}
\newcommand{\LD}{\mathrm{LD}}
\newcommand{\SD}{\mathrm{SD}}
\newcommand{\pat}{({\bf a},{\bf e})}
\begin{document}

\title[The spectrum of random lifts]{The spectrum of random lifts}
\author{L. Addario-Berry \and S. Griffiths}
\address{L.A-B. Department of Mathematics and Statistics, McGill University, 805 Sherbrooke Street West, 
		Montr\'eal, Qu\'ebec, H3A 2K6, Canada}
\email{louigi@math.mcgill.ca}
\date{December 17, 2010} 
\urladdrx{http://www.math.mcgill.ca/~louigi/}

\address{S.G. IMPA, Est.~Dona Castorina 110, Jardim Bot\^anico, Rio de Janeiro, Brazil}		
\email{sgriff@impa.br}


\begin{abstract} 
For a fixed $d$-regular graph $H$, a {\em random $n$-lift} is obtained 
by replacing each vertex $v$ of $H$ by a ``fibre'' containing $n$ vertices, 
then placing a uniformly random matching between fibres corresponding to 
adjacent vertices of $H$. We show that with extremely high probability, 
all eigenvalues of the lift that are not eigenvalues of $H$, have order $O(\sqrt{d})$. 
In particular, if $H$ is Ramanujan then its $n$-lift is with high probability nearly 
Ramanujan. We also show that any exceptionally large eigenvalues of the $n$-lift 
that are not eigenvalues of $H$, are overwhelmingly likely to have been caused 
by a dense subgraph of size $O(|E(H)|)$.
\end{abstract}

\maketitle


\section{Introduction}\label{sec:intro} 
Expander graphs, graphs in which all small subsets exhibit good expansion properties, are intriguing objects of study that arise in such diverse fields as number theory, computer science and discrete geometry.  As Hoory, Linial and Wigderson remark in their wide ranging survey on expanders \cite{HLW}, one reason for their ubiquity is that they may be defined in at least three languages: combinatorial/geometric, probabilistic and algebraic.  We refer the reader to this survey and to the expository article of Sarnak \cite{S} for more background on expander graphs and their applications.

We shall be concerned, almost exclusively, with {\em $d$-regular} graphs, in which every vertex has exactly $d$ neighbours.  
From the algebraic viewpoint, which we take throughout, a $d$-regular graph $G$ is an expander if there is a significant gap between $\lambda(G)$, the largest eigenvalue of the adjacency matrix of $G$ (for a $d$-regular graph this value is always $d$) and $\lambda_2(G)$, the largest modulus of any other eigenvalue.  
Classic results of Dodziuk, Alon--Milman and Alon \cite{D84,AM85,A} show that the difference $d-\lambda_2(G)$ controls the combinatorial expansion of $G$. More precisely, writing $h(G) = \min_{S} |E(S,S^c)|/|S|$, where $E(S,S^c)$ is the number of edges from $S$ to $S^c$, its complement in $V$, and where the minimum is taken over subsets $S$ of the vertices of $G$ with $|S| \leq |S^c|$, we have 
\[
\frac{d-\lambda_2(G)}{2} \leq h(G) \leq \sqrt{2d(d-\lambda_2(G))}.
\]
Given the theorem of Alon and Boppana \cite{A,N} that $\lambda_2(G) \ge 2\sqrt{d-1}-o_{n}(1)$ for every $d$-regular graph with $n$ vertices, it is particularly significant if $\lambda_2(G)\le 2\sqrt{d-1}$, in which case the graph is said to be {\em Ramanujan}.  A major open problem is to prove the existence of infinite families of $d$-regular Ramanujan graphs for all $d\ge 3$.   Explicit constructions coming from number theory given by Lubotzky, Phillips and Sarnak \cite{LPS} and Margulis \cite{Ma} for the case that $d-1$ is prime represented a major breakthrough.  Morgenstern \cite{Mo} gives examples of such families whenever $d-1$ is a prime power.  However, it seems unlikely that number theoretic approaches will be successful in resolving the problem in its full generality.

A combinatorial approach to the problem, initiated by \citet{F1}, is to prove that one 
may obtain new (larger) Ramanujan graphs from smaller ones. In this approach one starts with a base graph $H$ which one ``lifts'' to obtain a larger graph $G$ which covers the original graph $H$ in the sense that there is a homomorphism from $G$ to $H$ such that all fibres in $G$ of vertices of $H$ are of equal size. If $G$ is a cover of $H$ and the fibres in $G$ of vertices in $H$ have size $k$, then $G$ is called a $k$-lift of $H$. 

It is easily observed that the lift $G$ inherits all the eigenvalues of the base graph $H$.  Indeed, let $\mu$ be an eigenvalue of $H$ with eigenvector $x$, and define a vector $y$ with entries indexed by $V(G)$ by setting, for each $i$, $y_{v}=x_{i}$ for all vertices $v\in V_{i}$; 
then $y$ is an eigenvector of $G$ with eigenvalue $\mu$.  In fact these lifted eigenvectors of $H$ span the space of all vectors that are constant on each of the fibres $V_{i}$ of the lift.  The remaining eigenvalues of $G$ are referred to as the {\em new eigenvalues} of the lift (note however that it is possible for some new eigenvalues to be equal to ``old'' eigenvalues).  Since eigenvectors of symmetric real matrices corresponding to distinct eigenvalues are orthogonal, these are exactly the set of eigenvalues which have an eigenvector $x$ which is balanced on each fibre (i.e. for which $\sum_{v\in V_{i}}x_{v}=0$ for all $i\in V(H)$).  
Since the base graph is given it suffices to concentrate our study on the new eigenvalues of $G$.  We denote by $\lambda^{*}(G)$ the largest absolute value of a new eigenvalue of $G$. (For the remainder of the paper, 
for any graph $F$ we denote by $\lambda(F)$ the largest eigenvalue of $F$.) 

A {\em random $n$-lift} $G$ of a graph $H$ is obtained by assigning to each vertex $i$ of $H$ a distinct set $V_{i}$ of $n$ vertices, and placing a random matching (i.e. one chosen uniformly at random from the $n!$ possibilities) between $V_{i}$ and $V_{j}$ for each edge $ij$ of $H$.  Random lifts were introduced by Amit, Linial, Matou\v sek and Rozenman \cite{ALMR}. In that article a variety of properties of random lifts are discussed, related to connectivity, expansion, independent sets, colouring and perfect matchings; the proofs of these results, and others, were developed in several subsequent papers \cite{AL1,AL2,AL3,LR}. 
As remarked in \cite{F1}, any finite cover $F$ of $H$ in which fibres have size $n$ has a positive probability of appearing as $G$ (in fact this probability is precisely $(n!)^{|V(H)-E(H)|}(\mathrm{Aut}(F/H))^{-1}$, where $\mathrm{Aut}(F/H)$ is the group of automorphisms of 
$F$ over $H$), and so such random lifts form a ``seemingly reasonable model of a probabilistic space of finite quotients of [the infinite 
$d$-ary tree]''. 

Although very few graphs are {\em known} to be Ramanujan, it is conjectured that a positive proportion of regular graphs are in fact Ramanujan, and 
Alon's conjecture/Friedman's theorem states that for any $\eps > 0$, only an asymptotically negligible proportion of $d$-regular graphs have $\lambda_2(G) > 2\sqrt{d-1}+\eps$. In this spirit, Friedman \cite{F1} studied the eigenvalues of random lifts of regular graphs, and Lubetzky, Sudakov and Vu \cite{LSV} conjectured that a random lift of a Ramanujan graph has a positive probability of being Ramanujan.  Since for all $d$ the complete graph $K_{d+1}$ is a $d$-regular Ramanujan graph, this would imply the existence of arbitrarily large $d$-regular Ramanujan graphs.  In the terminology from \cite{F1}, the main result of this paper implies that with extremely high probability the lifts of Ramanujan graphs are $O(\sqrt{d})$-weakly Ramanujan, in that all non-trivial eigenvalues are $O(\sqrt{d})$.

In \cite{F1}, Friedman used the trace method of Wigner to prove results which in particular imply that if $H$ is $d$-regular then $\lambda^{*}(G)= O(d^{3/4})$ whp\footnote{If a statement holds with probability which tends to $1$ as $n$ tends to infinity, we say that it occurs `with high probability', or `whp' for short.}.  This was later tightened to $O(d^{2/3})$ by Linial and Puder, by a careful analysis of the trace method.  They also made a conjecture concerning word maps, which if verified would prove $\lambda^{*}(G)=O(d^{1/2})$ whp.  Two initial cases of the conjecture were proved, a third has been proven more recently by Lui and Puder \cite{LP}.  
Bilu and Linial \cite{BL} then showed that every $d$-regular graph $H$ has {\em some} $2$-lift $G$ with $\lambda^*(G) = O(d^{1/2}\log^{3/2}{d})$.
The next major step in this area was taken by Lubetzky, Sudakov and Vu \cite{LSV}, who proved that $\lambda^{*}(G)=O(\max(\lambda_2(H),d^{1/2})\log{d})$ whp.  In particular, in the case that $\lambda_2(H)=O(d^{1/2})$ this result gives that $\lambda^{*}(G)=O(d^{1/2}\log{d})$. 

In this article we prove that whp $\lambda^*(G)=O(d^{1/2})$, a result which is best possible up to the constant.  
\begin{thm}\label{main}
Let $H$ be any $d$-regular graph and let $G$ be a random $n$-lift of $H$.
For all $n$ sufficiently large, 
with probability at least $1-n^{-2d^{1/2}}$, 
$\lambda^*(G) \leq 430656 \sqrt{d}$. 
\end{thm}
Furthermore, we are able to {\em explain} the likely cause of large eigenvalues should they occur.  This cause is, with very high probability, a small (i.e.~of size not depending on $n$) subgraph of $G$.
\begin{thm}
\label{explain}
Let $H$ be any $d$-regular graph, write $h=|V(H)|$, and let $G$ be a random $n$-lift of $H$.
For all $n$ sufficiently large, with probability at least $1-n^{-hd}$, $G$ contains an induced subgraph $G'$ with at most $hd$ vertices, 
such that $\lambda^*(G) \leq 1189248 \lambda(G')$. 
\end{thm}
One might protest that the eigenvalue of $G'$ is not necessarily a new eigenvalue of $G$, and so $G'$ is not the `cause' of a new eigenvalue of large modulus in $G$.  However, the following approximate converse to \refT{explain} justifies our use of such an epithet for $G'$.
\begin{prop}\label{protest}
For any induced subgraph $G'$ of $G$ with $|V(G')| \leq n-h\sqrt{n}$, we have $\lambda^*(G) \geq \lambda(G') - 7/2$. 
\end{prop}
The short proof of \refP{protest} appears in \refS{sec:zbound}.

Costello and Vu \cite{CV} remark that ``[t]he main intuition that underlies many problems concerning the rank of a random matrix is that dependency should come from small configurations,'' and their paper can be seen as confirmation of this intuition for a rather broad class of random matrices. In this spirit, \refT{explain} should be viewed as stating that for random lifts, {\em any exceptionally large eigenvalues come from small configurations.}
It would be very interesting to know whether our ``exceptionally large'' can be replaced by ``slightly large''. A rather ambitious question one could ask in this direction is the following.
\begin{question}
Is there a constant $C$ not depending on $n$ (perhaps depending on $H$) such that for any 
$\epsilon > 0$, given that $\lambda^*(G) \geq (2+\eps)\sqrt{d-1}$, with probability $1-o_n(1)$, $G$ contains 
a subgraph $G'$ with at most $C$ vertices such that $\lambda(G') \geq (2+\eps/2)\sqrt{d-1}$?
\end{question}

We note that the probability bound in \refT{explain} is extremely strong.  Indeed, the failure probability, $n^{-hd}$  is much smaller than the probability that $G$ contains $H$ as a subgraph -- the probability 
of the latter event is greater than $n^{-hd/2}$ -- in which case $\lambda^{*}(G)=d$.

Our proof, like that of Lubetzky, Sudakov and Vu \cite{LSV} and many others, relies on reducing an uncountable collection of possible `reasons' for a large eigenvalue to a finite (and hopefully relatively small) sub-collection which still express all ways in which a large eigenvalue can occur.  They used the well-known method of $\epsilon$-nets to make this reduction. (Amit and Linial \cite{AL2} 
also used $\epsilon$-nets to prove a lower bound on edge expansion for random lifts of connected graphs which need not necessarily be regular.)  However, the number of events (points of the $\epsilon$-net) one is required to consider is $\exp(\Theta(nh\log{d}))$.  The appearance of the $\log{d}$ here is a major obstacle to proving that $\lambda^*(G)=O(\sqrt{d})$ by the $\epsilon$-net approach.  Our approach is based on a convexity argument that allows us to reduce to a smaller, $\exp(\Theta(nh))$-sized family of events.  Furthermore, the events in this collection are concerned with vectors with dyadic entries.  These are easier to deal with than general vectors; in particular, direct combinatorial arguments may be applied and we need not appeal to martingale inequalities to obtain probability bounds.

Modulo a few trivial changes to our proof (e.g. changing ``precisely'' to ``at most'' in the proof of Proposition \ref{zbound}, below) replacing $d$ by $\Delta$ throughout gives a proof of the following generalisation to the case that the base graph $H$ is not regular. 

\begin{thm} \label{thm:omit} Let $H$ be any graph of maximum degree $\Delta$ and let $G$ be a random $n$-lift of $H$.  For all $n$ sufficiently large, 
with probability at least $1-n^{-2\Delta^{1/2}}$, 
$\lambda^*(G) \leq 430656 \sqrt{\Delta}$, and with probability at least $1-n^{-h\Delta}$, 
$G$ contains an induced subgraph $G'$ with at most $h\Delta$ vertices, 
such that $\lambda^*(G) \leq 1189248 \lambda(G')$. \end{thm}  

Since $G$ will always contain two edge disjoint stars whose centres have degree $\Delta$ and lie in the same fibre, $\lambda^*(G)\ge \sqrt{\Delta}$.  Thus this result is also tight up to the constant. 


While we focus here on the case of an $n$-lift where $n$ is large we would like to bring to the reader's attention the 
recent result of Oliveira \cite{Olive} that a random $n$-lift $G$ of a graph $H$ on $h$ vertices with maximum degree $\Delta$ satisfies $\lambda^{*}(G)=O(\Delta^{1/2}\log^{1/2}(hn))$ whp.

\section{Notation}

All logarithms in this paper are natural logarithms unless otherwise specified. 
For positive integers $k$ we write $[k]=\{1,\ldots,k\}$. We write $\N_0 = \{0,1,2,\ldots\}$ and $\N=\{1,2,\ldots\}$. 
For any graph $F=(V,E)$ and $u,v \in V$ we write $u \sim_F v$ if $uv \in E$. 
For the remainder of the paper, $d \geq 2$ is a positive integer, $H=([h],E(H))$ is a fixed $d$-regular graph, 
and $G=(V(G),E(G))$ is the random $n$-lift of $H$, where $V(G)=\{(i,j),i \in [h],j \in [n]\}$. 
For $i$ in $[h]$ we write $V_i =  \{(i,j): j \in [n]\}$, and call $V_i$ the {\em fibre} of $i$ in $G$. 
Let $M$ be the adjacency matrix of $G$. 

For any $m$-by-$m$ matrix $A=(a_{uv})_{u,v \in [m]}$ and vectors $x,y \in \R^m$, we write 
$\ip{x}{y}_A = \sum_{u,v \in [m]} x_u a_{uv} y_v$. Also, for a set $E \subset \R^2$, we write 
$\ip{x}{y}_{A,E}$ for the restricted sum
\[
\sum_{\{u,v \in [m]: (x_u,y_v) \in E\}} x_u a_{uv} y_v. 
\]
Finally, we use the Vinogradov notation $f \ll g$ to mean that  $f=O(g)$, i.e.~$f$ is bounded by a constant times $g$, 
independent of $n$. We write $f \asymp g$ to mean that $f=O(g)$ and $g=O(f)$. 
               
\section{An overview of the proof} \label{sec:duck}  
As in \cite{KS} and much subsequent work, we will bound the eigenvalues of $G$ using the 
Rayleigh quotient principle, which is to say by bounding $\ipo{x}_M$ for suitable vectors~$x$. 
More precisely, writing $X = \{x \in \R^{V(G)}:\|x\|_2^2\leq 1\}$, Rayleigh's quotient principle tells us the following. 
\begin{fact}\label{rayleigh}
$\lambda^*(G)  = \sup\{ |\ipo{x}_M|:  x \in X, \forall~i \in [h],~\sum_{v \in V_i} x_v = 0\}$. 
\end{fact}
\refF{rayleigh} forms the basis for our study of $\lambda(G)$. 
However, to make use of it we first need to have an idea 
of the diversity of possible ways in which $G$ could admit a balanced vector 
(a vector satisfying 
$\sum_{v \in V_i} x_v = 0$ for all $i \in [h]$) having a large value for $\ipo{x}_M$.
To begin to get a feel for this, we now give two rather different examples of how $\lambda^*(G)$ can be large.  For simplicity, for the examples we assume that $H$ is the complete graph $K_{d+1}$.  We also provide bounds on the probability of such examples occurring in $G$.  These bounds give something of the flavour of the bounds we shall be required to prove for the general case.

\textbf{Example 1: $G$ contains a large-ish clique.}

Fix $s\in\N$ and vertices $v_1,\ldots,v_s$ of $G$ in distinct fibres. If $G[\{v_1,\ldots,v_s\}]$ happens to be a clique, then setting $x_{u}$ to be $1$ if $u \in \{v_1,\ldots,v_s\}$, $-1/(n-1)$ for all other vertices in the same fibres as $v_1,\ldots,v_s$, and zero on all other vertices, we obtain a vector $x$, balanced on each fibre of $H$ and for which $Mx=(s-1)x$. Thus, if $G$ contains a clique of order at least $K\sqrt{d}+1$ then $\lambda^*(G) \geq K\sqrt{d}$. 

To bound the probability that $G$ contains such a clique, for each possible choice of $v_1,\ldots,v_s$ the probability that $G[\{v_{1},\dots ,v_{s}\}]$ is a clique is $n^{-\binom{s}{2}}$ (since each edge $v_{i}v_{i'}$ is present in $G$ independently with probability $n^{-1}$).  Since there are only ${(d+1) \choose s} n^{s}$ choices of the $s$-tuple $(v_{1},\dots ,v_{s})$ the probability that $G$ contains a clique of size $s$ is at most ${(d+1) \choose s}n^{s-\binom{s}{2}}$, which is $o(1)$ for any $s\ge 4$. 

\textbf{Example 2: uneven edge densities all over $G$.} 

Suppose that there exist sets $(A_{i})_{i \in V(H)}$, with $A_{i}\subset V_{i}$ and $|A_{i}|=n/2$ for all $i\in V(H)$, such that $e(A_{i},A_{j})\ge n/4+Kn/\sqrt{d}$ for each $i \neq j$, $i,j \in V(H)$.  Then setting $x_{u}$ to be $(nh)^{-1/2}$ if $u$ is in $\bigcup_{i}A_{i}$ and $-(nh)^{-1/2}$ otherwise, we obtain a balanced vector $x$ with $\|x\|_2^2 = 1$ and $\ipo{x}_M \geq 2K\sqrt{d}$, so $\lambda^*(G) \geq 2K\sqrt{d}$.

Here, for each choice of sets $(A_{i})_{i\in V(H)}$, with $A_{i}\subset V_{i}$ and $|A_{i}|=n/2$ for all $i\in V(H)$, the probability that for all $i,j \in V(H)$, $i\neq j$ we have $e(A_{i},A_{j})\ge n/4+Kn/\sqrt{d}$, is of order at most $\exp(-c\binom{d+1}{2}K^{2}n/d)$, for some constant $c > 0$ (this is not hard to derive by hand; it can also be obtained straightforwardly from \refP{bigbound} in Section~\ref{sec:prob}). If $K$ is sufficiently large, this bound is strong enough for a union bound to show
that with high probability, there is no such choice of sets $(A_i)_{i \in V(H)}$. 

The preceding examples present two rather different structures within $G$, both of which give rise to large new eigenvalues, 
and show how the new eigenvectors are also rather different. 
We may also switch our point of view, first fixing a vector $x$ (for the first example a vector taking the value $1$ on a single vertex of each fibre $V_{i}, \, i=1,\dots,s$ and taking the value $-\I{u\in V_{[s]}}/(n-1)$ on other vertices $u$; for the second example a vector taking the value $(nh)^{-1/2}$ on $n/2$ vertices in each fibre and $-(nh)^{-1/2}$ on the rest) then asking what structure in $G$ is required if $|\ipo{x}_M|$ is to be large for this specific vector $x$. This is essentially the perspective we will take for most of the rest of the paper. 

From this viewpoint, a possible cause for $\lambda^{*}(G)$ being large consists of a vector $x \in X$ together with evidence, in the form of specified edge counts between subsets of vertices of $G$, that $|\ipo{x}_M|$ is large. 
For a vector $x$, $i \in [h]$ and $w \in \R$ and let
\[
 A_{i,w}(x) = \{v \in V_i: x_{v} = w\}, \qquad a_{i,w}(x)=|A_{i,w}(x)|. 
\]
We write $A_{i,w}=A_{i,w}(x)$ and $a_{i,w}=a_{i,w}(x)$ when the dependency on $x$ is clear. 
We say that the collection $\{a_{i,w}:i \in [h],w \in \R\}$ is the {\em type} of the vector $x$, and denote this collection ${\bf a}(x)$. By the symmetry of the model, the probability that $|\ipo{x}_M|$ 
is large should be a function only of the type of $x$. We will seek sets of constraints on the edge densities between sets corresponding to distinct $a_{i,w}$ and $a_{i',w'}$, which 
codify all possible ways in which a large new eigenvalue can appear. A type, together with a particular such set of constraints, will be called a {\em pattern}; the precise definition of patterns will appear later in the section. 

While this initial concept of a pattern is useful for demonstrating the idea we have in mind, a number of changes are needed before we can 
put it into play. The most fundamental issue is that we are trying to bound $\sup_{x \in X} \ipo{x}_M$, the supremum being over an uncountable collection. 
We will shortly show that for a moderate cost, the uncountable 
supremum in \refF{rayleigh} can be replaced by a supremum over a more tractable collection. 
Also, it turns out that for {\em any} vector $x \in X$, the contribution to $\ipo{x}_M$ made by entries $x_{i,j},x_{i',j'}$ whose 
weights differ by more than a factor of $\sqrt{d}$ is negligible, which will allow us to further restrict the collection of types we need to consider. 
This is not a complete list of the required changes, but before proceeding too far into the argument it is useful to fill in some of these initial steps. 

Recall that $M$ is the adjacency matrix of $G$, 
and let $\overline{M} = (\overline{m}_{(i,j),(i',j')})_{(i,j),(i',j') \in V(G)}$ be the matrix with 
\[
\overline{m}_{(i,j),(i',j')} = \begin{cases}
																	1/n	& \mbox{if } i \sim_{H} i' \\
																	0							& \mbox{otherwise.}
															\end{cases}
\]
In other words, $\overline{m}_{(i,j),(i',j')} = \p{(i,j) \sim_G (i',j')}$, so 
for all $x \in \R^{V(G)}$, we have $\ipo{x}_{\overline{M}} = \E{\ipo{x}_M}$. 
If $x$ is balanced on each fibre, i.e., satisfies $\sum_{v \in V_i} x_v = 0$ for all $i \in [h]$, 
then $\ipo{x_i}_{\overline{M}} =x_i^t (\overline{M}x_i) = x_i^t \cdot \overline{0}=0$. 
Thus, letting $N=M-\overline{M}$, for such $x$ we have 
\[
\ipo{x}_N = \ipo{x}_M - \ipo{x}_{\overline{M}} = \ipo{x}_M. 
\]
If, on the other hand, $x$ is constant on fibres of $G$, then $\ipo{x}_M = \ipo{x}_{\overline{M}}$ and so 
$\ipo{x}_N=0$. 
From this we easily obtain the following fact. 
\begin{fact}\label{uncount}
$\lambda^*(G) = \sup_{x \in X} |\ipo{x}_N|$.
\end{fact}
\begin{proof}
The inequality $\lambda^*(G) \leq \sup_{x \in X} |\ipo{x}_N|$ follows from \refF{rayleigh} 
since $\ipo{x}_N=\ipo{x}_M$ if $x$ is balanced on fibres of $G$. 
On the other hand, any vector in $x \in X$ can be expressed as $y+z$, where $y \in X$ is balanced on fibres of $H$ 
and $z$ is constant on fibres of $H$. Since $\langle y,z \rangle_N=0$, we then have $\ipo{x}_N=\ipo{y}_N=\ipo{y}_M$, which proves the other inequality. 
\end{proof}
By using $N$ instead of $M$, \refF{uncount} allows us to maintain the property of only considering ``new'' eigenvectors of $M$, without 
insisting that the vectors we consider remain balanced on each fibre of $G$. This turns out to be remarkably helpful. 

Next, let $D^+ = \{0\} \cup \{2^i/(nh)^{1/2}: i \in \N\}$, let 
\begin{align}
Z 	& = \left\{x \in \R^{V(G)}:\|x\|_2^2 \leq 10,\forall~v \in V(G),x_v \in D^+, \sup_{v \in V(G)} x_v \leq d \cdot \inf_{v \in V(G): x_v \ne 0} x_v\right\}\, . \label{zdef}
\end{align}
and let 
\[
E^* = \{(x_1,x_2) \in \mathbb{R}^2: x_1,x_2 > 0, x_1/x_2 \in (d^{-1/2},d^{1/2})\}. 
\]
The following proposition, which is proved in \refS{sec:zbound}, formalizes our discussion on ``restricting the collection of types we need to consider''. 
\begin{prop}\label{zbound}
$\lambda^*(G) \leq 96 \sup_{x \in Z} |\ipo{x}_{N,E^*}| + 480\sqrt{d}.$
\end{prop}

A relatively straightforward step in the proof of \refP{zbound}, that will also turn out to be useful in another part of the paper, is to show that not much is lost when we restrict our attention to pairs of vertices whose weights in $x$ differ by at most a multiplicative factor of $\sqrt{d}$. This is the content of the following lemma, which is proved in \refS{sec:zbound}. 
\begin{lem}\label{ipoe}
For all $y \in \R^{V(G)}$, 
$|\ipo{y}_{N,\R^2\setminus E^{*}}|\le 4\sqrt{d}\|y\|_2^2.$
\end{lem}

We are now in a position to further elaborate on what will constitute a pattern of a large eigenvalue. 
Recall that for $y \in \R^{V(G)}$, 
for each $i \in [h]$ and each $w \in \R$, we have $A_{i,w}(y) = \{v \in V(G):y_{v} = w\}$, 
and $a_{i,w}(y) = |A_{i,w}(y)|$, and that ${\bf a}(y)$ is called the type of $y$.  
We remark that if $y \in Z$ then ${\bf a}(y)$ satisfies the following properties.

\begin{enumerate}
\item[{\bf 1.}] For all $w \in \R$, $w \not \in D^+$, we have $a_{i,w} = 0$.
\item[{\bf 2.}] For all $i \in V(H)$, $\sum_{w\in D^+} a_{i,w} \le n$. 
\item[{\bf 3.}] There exists $w_{0}\in D^+$, $w_0>0$, such that $a_{i,w}=0$ unless $w_{0}\le w\le w_{0}d$. 
\item[{\bf 4.}] $\sum_{i \in [h], w \in D^+} w^2 a_{i,w} \le 10$. 
\end{enumerate}
We call a collection ${\bf a} = \{a_{i,w}: i \in V(H), w \in \R\}$ a {\em $Z$-type} if it satisfies properties {\bf 1}-{\bf 4}, so that 
if $y \in Z$ then the type ${\bf a}(y)$ of $y$ is a $Z$-type. 
Next, given $y \in \R^{V(G)}$ and $ii' \in E(H)$, $w,w' \in \R$, let $e_{i,w,i',w'}(y,G) = |e_G(A_{i,w}(y),A_{i',w'}(y)|$, and write 
\[
{\bf e}(y,G) = \{e_{i,w,i',w'}(y,G): ii' \in E(H), w,w' \in \R\}.
\]
For all $ii' \in E(H)$ and all $w,w' \in \R$, we necessarily have 
\begin{equation*}\label{eq:edef}
e_{i,w,i',w'}(y,G) \in \{0,1,\ldots, \min(a_{i,w}(y),a_{i',w'}(y))\}.
\end{equation*}
A {\em pattern} is a pair $({\bf a},{\bf e})$, where ${\bf a}$ is a $Z$-type and 
\[
{\bf e}= \{e_{i,w,i',w'}: ii' \in E(H), w,w' \in \R\}
\]
is a collection with $e_{i,w,i',w'} \in \{0,1,\ldots, \min(a_{i,w},a_{i',w'})\}$ for all $ii' \in E(H)$ and all $w,w' \in \R$. 

Write $\dog=D^+\setminus\{0\}=\{2^i/\sqrt{nh}:i \in \N\}$, and 
let $\Gamma$ be the graph on vertex set $\{(i,w):i\in V(H), w\in\dog\}$ with an edge $(i,w)(i',w')$ if $i\sim_H i'$ and $w/w'\in (d^{-1/2},d^{1/2})$.

For a given pattern, $({\bf a},{\bf e})$, we write 
\begin{align*}
p({\bf a},{\bf e}) = \left|\sum_{(i,w)\simg (i',w')}w w' \left(e_{i,w,i',w'}- \frac{a_{i,w}a_{i',w'}}{n} \right)  \right| ,
\end{align*}
and call $p({\bf a},{\bf e})$ the \emph{potency} of $({\bf a},{\bf e})$. We remark that for any $y \in Z$, 
\begin{align}
|\langle y,y \rangle_{N,E^*}| & = 2\left|\sum_{(i,w)\simg (i',w')}w w' \left(e_{i,w,i',w'}(y,G)- \frac{a_{i,w}(y)a_{i',w'}(y)}{n} \right)\right| \nonumber\\
& =  2\,  p({\bf a}(y),{\bf e}(y,G)), \label{eq:potconnect}
\end{align}
the factor $2$ arising from the symmetry of the matrix $N$. 
We say that a pattern $({\bf a},{\bf e})$ {\em can be found in $G$} if there exists $y \in Z$ such that 
${\bf a}(y)={\bf a}$ and ${\bf e}(y,G)={\bf e}$. 
In this case we say that $G$ {\em contains} the pattern $({\bf a},{\bf e})$, 
and call the collection $\{A_{i,w}(y)\}_{(i,w) \in V(\Gamma)}$ a {\em witness} for the pattern $({\bf a},{\bf e})$. 

\begin{prop}\label{findwit}
For $K > 0$, if $\lambda^*(G) \geq 192(K+3) \sqrt{d}$ then a pattern $({\bf a},{\bf e})$ with potency at least $K\sqrt{d}$ can be found in $G$.
\end{prop}
\begin{proof}
By \refP{zbound}, in this case there exists $y \in Z$ with $|\langle y,y \rangle_{N,E^*}| \ge 2K\sqrt{d}$, and the result follows from (\ref{eq:potconnect}). 
\end{proof}
Our main aim, then, is to show that with high probability, no high-potency pattern can be found in $G$. 
To do so, we shall use a reduction which is conceptually analogous to how one bounds the probability that a fixed graph $H$ appears as a subgraph of the Erd\H{o}s--R\'enyi random graph $G(n,p)$: rather than proving the bound directly for $H$, one instead considers a strictly balanced (maximally edge-dense) subgraph of $H$. 

Given a pattern $({\bf a},{\bf e})$ and a set of vertices $S \subset V(\Gamma)$, we define the 
{\em sub-pattern of $({\bf a},{\bf e})$ induced by $S$} to be the pattern $({\bf a}',{\bf e}')$ obtained from $({\bf a},{\bf e})$ by 
setting 
\[
a'_{i,w} = \begin{cases} 
					a_{i,w} 	& \mbox{if}~(i,w) \in S \\
					0				& \mbox{otherwise}. 
				\end{cases}
\]
and setting 
\[
e_{i,w,i',w'}' = \begin{cases} 
					e_{i,w,i',w'} 	& \mbox{if}~(i,w),(i',w') \in S \\
					0				& \mbox{otherwise}. 
			\end{cases}
\]
We write $({\bf a},{\bf e})_S$ for the sub-pattern of $({\bf a},{\bf e})$ induced by $S$.  

We also require the following variant of potency, which 
allows us to consider a ``maximally potent'' subgraph of $\Gamma$ rather than $\Gamma$ itself. For a pattern $\pat$ we define
\[
\tilde{p}(\pat):=\max_{E\subset E(\Gamma)} \left|\sum_{(i,w)(i',w')\in E}w w' \left(e_{i,w,i',w'}- \frac{a_{i,w}a_{i',w'}}{n}\right) \right|
\]
We will prove the following theorem.
\begin{thm}\label{redbound}
Fix $L \ge 20$. 
For any pattern $({\bf a},{\bf e})$, there exists $S=S(({\bf a},{\bf e}),L) \subset V(\Gamma)$ such that the following properties hold.
\begin{align*}
\tilde{p}(({\bf a},{\bf e})_{S}) 	& \geq \frac{p({\bf a},{\bf e})}{2}-55L\sqrt{d}\\
\p{({\bf a},{\bf e})_S~\mbox{can be found in G}} & \ll \prod_{(i,w) \in S} a_{i,w}^{d/4}\pran{\prod_{(i,w) \in S} {n \choose a_{i,w}\wedge \lfloor n/2 \rfloor}}^{1-L/10}. 
\end{align*}
\end{thm}

We prove this theorem in \refS{sec:key}. 
\refT{main} will follow straightforwardly from \refP{findwit} and \refT{redbound}.
More precisely, \refP{findwit} ensures that the number of events whose probability we must bound is not too large, and \refT{redbound} ensures that the probability of each event is sufficiently small that we can simply apply a union bound to prove \refT{main}. 
Before {\em providing} the proof, we state one additional, easy bound which we will require, on 
the total number of patterns satisfying a constraint on the sizes of the $a_{i,w}$.
\begin{lem}\label{patcount}
For fixed $A \in \N$, the number of patterns with $a_{i,w} < A$ for all $(i,w) \in V(\Gamma)$ 
is at most $\log_2(nh) \cdot A^{2h d\log_2 d}$. 
\end{lem}

Assuming this lemma, which is proved in \refS{sec:zbound}, we are now ready to give our proof of \refT{main}. (The proof of \refT{explain}, while conceptually almost identical to that of \refT{main}, ends up requiring a somewhat more technical development and we therefore defer it to \refS{sec:explain}.)

\begin{proof}[Proof of Theorem~\ref{main}] 
Fix $M\, \ge\, 430656\,\, = \,\, 192(2240\, +\, 3)$, let $K=K(M)=(M/192)-3 \ge 2240$, and let $L(M)=K/112\ge 20$.

By \refP{findwit}, if $\lambda^*(G) \geq M\sqrt{d}$ then there is a pattern $\pat$ with potency at least 
$K\sqrt{d}=112 L\sqrt{d}$ such that $\pat$ can be found in $G$.   Let $\pat_{S(\pat,L)}$ be the sub-pattern of $\pat$ obtained by applying \refT{redbound} to $\pat$, it follows that $\tilde{p}(\pat_{S(\pat,L)}) \ge (112/2-55)L\sqrt{d} =L\sqrt{d}$ and that $\pat_{S(\pat,L)}$ can be found in $G$.

We say a pattern $\pat$ is an {\em $L$-reduction} if there is a pattern $({\bf a}',{\bf e}')$ and 
$S=S(({\bf a}',{\bf e}')),L)$ as in \refT{redbound}, such that $({\bf a}',{\bf e}')_{S(({\bf a}',{\bf e}'),L)}=\pat$. It follows from the preceding paragraph that for any $M \geq 430656$, 
\[
\p{\lambda^*(G) \geq M\sqrt{d}} \leq \p{E},
\]
where $E=E(M)$ is the event that $G$ contains an $L(M)$-reduction $\pat$ with $\tilde{p}\pat\ge L(M)\sqrt{d} \geq 20\sqrt{d}$.  We use \refT{redbound} to bound this probability for each fixed pattern (reduction) $\pat$, the proof is then completed with a union bound.  We now give the details.  We split into two cases, depending on whether the pattern (reduction) $\pat$ has $a_{i,w}\geq 4hd\log_{2}{d}$ for some $(i,w) \in V(\Gamma)$.  Let $E_{1}=E_1(M)$ be the event that $G$ contains an $L$-reduction $\pat$ with $\tilde{p}\pat\ge L\sqrt{d}$ in which at least one $a_{i,w}\geq 4hd\log_{2}{d}$ and $E_{2}=E_2(M)$ the event that $G$ contains an $L$-reduction $\pat$ with $\tilde{p}\pat\ge L\sqrt{d}$ in which $a_{i,w}< 4hd\log_{2}{d}$ for all $i,w$.  We note that $E_1$ and $E_2$ need not be disjoint. To prove \refT{main} it suffices to show that both $\p{E_1}$ and $\p{E_2}$ are less than $n^{-2d^{1/2}}/2$ for all $n$ sufficiently large. 

First note that for {\em any} $L$-reduction $\pat$, writing $\alpha=\sum_{(i,w) \in V(\Gamma)} a_{i,w}$, since $1 - L/10 \le -1$, by \refT{redbound} we have 
\[
\p{\pat_S~\mbox{can be found in}~G} \ll \prod_{(i,w) \in S} a_{i,w}^{d/4}\cdot {n \choose \alpha \wedge \lfloor n/2 \rfloor}^{-1}. 
\]
For any $L$-reduction $\pat$ considered in $E_1$, we have $\alpha \geq 4hd \log_2(d)$. 
We note that since non-zero entries in any $Z$-type differ by a factor of at most $d$, for any $i \in V(H)$ there are at most $\log_2(2d)$ values $w\ne 0$ for which $a_{i,w} \ne 0$. It follows that in total there are at most $h\log_2(2d)$ vertices $(i,w) \in V(\Gamma)$ with $a_{i,w} \ne 0$, so
\begin{align*}
\p{\pat_S~\mbox{can be found in}~G} 
& \ll n^{(d/4) h\log_2(2d)} \cdot {n \choose 4hd\log_2(d)}^{-1} \\
& \leq n^{(hd\log_2 d)/2} \cdot \pran{\frac{8hd \log_2 d}{n}}^{4hd \log_2 d}.
\end{align*}
We next take a union bound over all $L$-reductions considered in $E_1$. 
Clearly the number of such reductions is at most the total number of patterns, which by \refL{patcount} applied with $A=n$ is at most $\log_2(nh) n^{2hd\log_2 d}$. 
It follows that 
\[
\p{E_1} \ll (8hd \log_2 d)^{4hd \log_2 d} \frac{\log_2(nh)}{n^{3(hd\log_2 d)/2}},
\]
and so $\p{E_1}\le n^{-2d^{1/2}}/2$ for $n$ sufficiently large. 

Next, for any $L$-reduction $\pat$ considered in $E_2$, we have 
\[
\prod_{(i,w) \in S} a_{i,w}^{d/4} < (4hd\log_2 d)^{(d/4)h\log_2 (2d)}.
\]
Also, for any pattern $\pat$, it follows straightforwardly from the definition of a pattern that 
\begin{align*}
\tilde{p}\pat 	& \leq  \sum_{(i,w)\simg (i',w')}w w' \left|e_{i,w,i',w'}- \frac{a_{i,w}a_{i',w'}}{n} \right|\\
 			& \leq \sum_{(i,w)(i',w') \in E(\Gamma)} ww' \min(a_{i,w},a_{i',w'}) \\
			& \leq \sum_{(i,w) \in V(\Gamma)} \mathop{\sum_{(i',w')\simg(i,w)}}_{w' \leq w,a_{i,w} >0} w^2 a_{i,w} \\
			& \leq \sum_{(i',w') \in V(\Gamma)} a_{i',w'} \cdot \sum_{(i,w) \in E(\Gamma)} w^2 a_{i,w} \\
			& \leq \sum_{(i,w) \in V(\Gamma)} a_{i,w}. 
\end{align*}
It follows that $\alpha=\sum_{(i,w) \in S} a_{i,w} \geq Ld^{1/2}> 3d^{1/2}$ for every $L$-reduction $\pat$ with $\tilde{p}\pat \ge L\sqrt{d}$.  So, for any pattern considered in $E_2$, for $n \geq 6d^{1/2}$ we have 
\[
{n \choose \alpha \wedge \lfloor n/2 \rfloor}^{-1} \leq \frac{(6d^{1/2})^{3d^{1/2}}}{n^{3d^{1/2}}}.
\]
Furthermore, the total number of $L$-reductions considered in $E_2$ is at most the total number of patterns 
with all $a_{i,w}$ less than $4hd \log_2d$, which by \refL{patcount} 
is at most $\log_2(nh) (4hd\log_2 d)^{2hd \log_2 d}$. 
Thus, by a union bound, 
\[
\p{E_2} \ll (4hd\log_2d)^{3hd\log_2 d}\cdot (6d^{1/2})^{2d^{1/2}}\cdot \frac{\log_2(nh)}{n^{3d^{1/2}}}, 
\]
which implies that $\p{E_2}\le n^{-2d^{1/2}}/2$ for $n$ sufficiently large. This completes the proof.
\end{proof}

\section{Two tools from probability} \label{sec:prob}
In this section, we establish all the probability bounds we will require in the remainder of the paper. \refP{bigbound}, below, 
bounds the probability that a random matching has certain prescribed edge counts between given pairs of sets. 
\refL{lem:measure} gives a lower bound on the integral of a function when its value is not too small relative to another function. 
We proceed immediately to details. 

Let $V=\{v_1,\ldots,v_n\}$, $W=\{w_1,\ldots,w_n\}$, and let $\bM$ be a uniformly random matching of $V$ and $W$. 
Let $A_0,\ldots,A_s$ and $B_0,\ldots,B_t$ be partitions of $V$ and $W$ respectively, and write $a_i=|A_i|$ and $b_j=|B_j|$. 
(Here $s$ and $t$ are constants not depending on $n$.) 
Writing $e_{\bM}(A_i,B_j)$ for the number of edges from $A_i$ to $B_j$ in $\bM$, it is straightforward that $\mu_{ij} = \E{e(A_i,B_j)} = a_ib_j/n$. 

Now fix integers $\{e_{ij}\}_{0 \leq i \leq s,0 \leq j \leq t}$ with $\sum_{j=0}^t e_{ij}=a_i$ for each $0 \leq i \leq s$ and $\sum_{i=0}^s e_{ij} = b_j$ for each $0 \leq j \leq t$. (From now on, summations and products over $i$ (resp.~$j$) should be understood to have $0 \leq i \leq s$ (resp.~$0 \leq j \leq t$).) We wish to bound 
\[
\p{\bigcap_{ij} \{e(A_i,B_j) = e_{ij}\}}.
\]
Our aim is to prove a bound not too different from what we would obtain were the $e(A_i,B_j)$ independent with Binomial$(a_ib_j,1/n)$ distribution. 
By direct enumeration, $\p{\bigcap_{ij} \{e(A_i,B_j) = e_{ij}\}}$ equals 
\[
\frac{1}{n!} \cdot \prod_i {a_i \choose e_{i0},\ldots,e_{it}} \prod_j {b_j \choose e_{0j},\ldots,e_{sj}} \prod_{ij} e_{ij}! = \frac{1}{n!}\prod_i a_i! \cdot \prod_j b_j! \cdot \pran{\prod_{ij} e_{ij}!}^{-1},
\]
with the convention that $0!=1$. Applying Stirling's formula, this is of the same order as 
\begin{equation}\label{jast}
\pran{\frac{\prod_i a_i \prod_j b_j}{\prod_{ij:e_{ij}\neq 0} e_{ij}}}^{1/2} \cdot n^{-(n+1/2)} \prod_{i} a_i^{a_i} \prod_{j} b_j^{b_j} \prod_{ij:e_{ij}\neq 0} e_{ij}^{-e_{ij}}
\end{equation}
Now write $e_{ij}=(a_ib_j/n)(1+\eps_{ij}) = \mu_{ij} (1+\eps_{ij})$. 
We then have 
\begin{align*}
\prod_{ij:e_{ij}\neq 0} e_{ij}^{e_{ij}} & = \prod_{ij:e_{ij} \neq 0} \mu_{ij}^{\mu_{ij}(1+\eps_{ij})} \prod_{ij:e_{ij}\neq 0} (1+\eps_{ij})^{\mu_{ij}(1+\eps_{ij})} \\
										& = \prod_{ij} \mu_{ij}^{\mu_{ij}(1+\eps_{ij})} \prod_{ij:e_{ij}\neq 0} (1+\eps_{ij})^{\mu_{ij}(1+\eps_{ij})} \\ 
										& = \prod_{ij} \mu_{ij}^{\mu_{ij}} \prod_{ij} \mu_{ij}^{\mu_{ij}\eps_{ij}} \prod_{ij:e_{ij}\neq 0} (1+\eps_{ij})^{\mu_{ij}(1+\eps_{ij})}  
\end{align*}
For fixed $i$, since $\sum_{j} e_{ij}=a_i$ we must have $\sum_j \eps_{ij} \mu_{ij}=0$. 
Likewise for fixed $j$ we have $\sum_i \eps_{ij} \mu_{ij} = 0$, and it follows that 
\begin{align*}
\prod_{ij} \mu_{ij}^{\eps_{ij} \mu_{ij}} & = \prod_{i} \pran{\prod_{j} \pran{\frac{a_i}{n}}^{\eps_{ij} \mu_{ij}} \prod_j b_j^{\eps_{ij} \mu_{ij}}} \\
											& = \prod_{ij} b_j^{\eps_{ij} \mu_{ij}} \\
											& = \prod_j \pran{ \prod_i b_j^{\eps_{ij} \mu_ij} } = 1.
\end{align*}
A similar calculation shows that 
\begin{align*}
\prod_{ij} \mu_{ij}^{\mu_{ij}} = n^{-n} \cdot \prod_i a_i^{a_i} \prod_j b_j^{b_j}, 
\end{align*}
and so \refQ{jast} equals 
\[
n^{-1/2}\pran{\frac{\prod_i a_i \prod_j b_j}{\prod_{ij:e_{ij}\neq 0} e_{ij}}}^{1/2} \cdot \prod_{ij:e_{ij}\neq 0} (1+\eps_{ij})^{-\mu_{ij}(1+\eps_{ij})}. 
\]
We now focus our attention on the latter product of the preceding equation. 
It is helpful to multiply and divide by $1= \prod_{ij} e^{\eps_{ij} \mu_{ij}}$, to obtain the equivalent expression 
\begin{align*}
	& \prod_{ij:e_{ij}\neq 0} \exp\pran{-\mu_{ij}(1+\eps_{ij})\log(1+\eps_{ij})} \prod_{ij} \exp\pran{\mu_{ij}\eps_{ij}} \\
= 	& \prod_{ij:e_{ij}\neq 0} \exp\pran{-\mu_{ij}[(1+\eps_{ij})\log(1+\eps_{ij})-\eps_{ij}]} \prod_{ij:e_{ij}=0} \exp\pran{-\mu_{ij}}.
\end{align*}
(For the second product we use that when $e_{ij}=0$, $\eps_{ij}=-1$.)
Since $(1+\eps)\log(1+\eps)-\eps$ approches $1$ as $\eps \downarrow -1$, taking $0\log0-0=1$ by convention then allows us to 
combine the two preceding products.
In sum, we have established the following proposition. 
\begin{prop}\label{bigbound}
Let $e_{ij} = \mu_{ij}(1+\eps_{ij})$ and let $E_{ij} = e_{\bM}(A_i,B_j)$. Writing $\chi = n^{-1/2}\pran{\frac{\prod_i a_i \prod_j b_j}{\cdot \prod_{ij:e_{ij}\neq 0} e_{ij}}}^{1/2}$, 
we have 
\[
\p{\bigcap_{ij} \{E_{ij} = e_{ij}\}} \asymp
\chi \cdot e^{-\sum_{ij} \mu_{ij}[(1+\eps_{ij})\log(1+\eps_{ij}) - \eps_{ij}]}.
\]
\end{prop}
In fact, it is a weakening of \refP{bigbound} that will be useful later in the paper. 
\begin{cor}
\label{cor:bb}
Under the conditions of \refP{bigbound}, we have 
\[
\p{\bigcap_{ij} \{E_{ij} = e_{ij}\}} \ll
\pran{\prod_{1 \leq i \leq s}a_i \prod_{1 \leq j \leq t} b_j}^{1/4} \cdot e^{-\sum_{ij} \mu_{ij}[(1+\eps_{ij})\log(1+\eps_{ij}) - \eps_{ij}]}.
\]
\end{cor}
\begin{proof}
For each $0 \leq i \leq s$ we have $\prod_{0 \leq j \leq t: e_{ij} \neq 0} e_{ij} \geq a_i$, 
so $\prod_{ij:e_{ij} \neq 0} e_{ij} \geq \prod_i a_i$. 
We likewise have $\prod_{ij:e_{ij} \neq 0} e_{ij} \geq \prod_j b_j$, and so 
\[
\prod_{ij:e_{ij} \neq 0} e_{ij} \geq \prod_i a_i^{1/2} \prod_j b_j^{1/2}. 
\]
Also, $n^{-1/2}a_0^{1/2}b_0^{1/2} \leq 1$. The result follows. 
\end{proof}
We will also use a lemma which is essentially a version of the following basic fact: 
if $X$ and $Y$ are non-negative random variables and $E$ is the event that $X \leq cY$, then $\E{X \I{E}} \le c\E{Y}$. 
(For generality we shall state the lemma in terms of a measure space, although we shall only apply it to finite measure spaces, and in this case the reader may think of the integrals as simply weighted sums.)

\begin{lem} \label{lem:measure} 
Let $(\Omega, \mathcal{E},\nu)$ be a measure space and $g,h:\Omega\to \mathbb{R}$ positive measurable functions, 
and fix $0 < c < 1$. 
Writing 
\[
E = \left\{ \omega \in \Omega: \frac{h(\omega)}{g(\omega)} \geq c\frac{\int h d\mu}{\int g d\mu}\right\},
\]
we have 
\[\int_E h d\mu \ge (1-c) \int_{\Omega} h d\mu.
\] 
\end{lem}
\begin{proof} \begin{equation*} \int_{\Omega \setminus E}hd\mu =\int _{\Omega\setminus E}g \cdot \frac{h}{g}d\mu 
\le c \int_{\Omega\setminus E}g \cdot \frac{\int h}{\int g}  d\mu
\le c\int_{\Omega} h d\mu \end{equation*}\end{proof}

\section{Reduced patterns: a proof of \refT{redbound}} \label{sec:key}
\refT{redbound} requires us to find, given a high potency pattern $({\bf a},{\bf e})$, a sub-pattern $({\bf a},{\bf e})_S$ which is unlikely to occur in $G$. In fact, our aim will be slightly more exigent. We will show the existence of sub-pattern $({\bf a},{\bf e})_S$ which is `locally unlikely' at all $(i,w)\in S$. Given a pattern $\pat$, for $i\sim_{H} i'$ and $w,w' \in \R$ write $\mu_{i,w,i',w'} = a_{i,w}a_{i',w'}/n$ and write 
\[
\eps_{i,w,i',w'}=\eps_{i,w,i',w'}\pat = 
\begin{cases}
1	& \mbox{ if }a_{i,w}=0\mbox{ or }a_{i',w'}=0,\\
e_{i,w,i',w'} / \mu_{i,w,i',w'} -1 & \mbox{ if } a_{i,w}\neq 0\mbox{ and }a_{i',w'} \neq 0. 
\end{cases}
\]
The quantity $\eps_{i,w,i',w'}$ encodes the deviation from the expected number $\mu_{i,w,i',w'}$ of edges between sets $A \subset V_i$, $A' \subset V_{i'}$ with $|A|=a_{i,w}$ and $|A'|=a_{i',w'}$ if $|E_G(A,A')|=e_{i,w,i',w'}$. We will bound the likelihood of such deviations using \refC{cor:bb}, and to that end define a function $b:(-1,\infty) \to [0,\infty)$
\[ b(\eps)= (1+\eps)\log(1+\eps) - \eps \qquad \eps\in (-1,\infty).\]
We then have the following. 
\begin{prop}\label{prop:redbound} Fix $L \ge 20$. 
For any pattern $({\bf a},{\bf e})$, there exists $S=S(({\bf a},{\bf e}),L) \subset V(\Gamma)$ such that the following properties hold.
\begin{align*}
\tilde{p}(({\bf a},{\bf e})_{S}) 	& \geq \frac{p({\bf a},{\bf e})}{2}-55L\sqrt{d}\qquad \text{and}\\
\sum_{(i',w')\in N_{\Gamma}(i,w)\cap S} \frac{a_{i,w}a_{i',w'}}{n} b(\eps_{i,w,i',w'})& \ge  \frac{L}{10} a_{i,w}\log\left(\frac{en}{a_{i,w}}\right) \qquad \text{for all} \, (i,w)\in S\, .\end{align*}
\end{prop}
The second inequality in \refP{prop:redbound} encodes the idea that the pattern is everywhere locally unlikely.
Before continuing to the proof of \refP{prop:redbound} let us show that \refT{redbound} does indeed follow from \refP{prop:redbound} and the probability bound \refC{cor:bb}.

\begin{proof}[Proof of \refT{redbound}]  
Given $L\ge 20$ and a pattern $({\bf a},{\bf e})$, let $S=S(({\bf a},{\bf e}),L)$ be the subset of $V(\Gamma)$ given by \refP{prop:redbound}.  Since it is immediate from the statement of \refP{prop:redbound} that 
\[ \tilde{p}(({\bf a},{\bf e})_{S})  \geq \frac{p({\bf a},{\bf e})}{2}-55L\sqrt{d},\]
all that is left to prove is that
\[\p{({\bf a},{\bf e})_S~\mbox{can be found in G}}  \ll \prod_{(i,w) \in S}\!\! a_{i,w}^{d/2}\, \pran{\prod_{(i,w) \in S} {n \choose a_{i,w}\wedge \lfloor n/2 \rfloor}}^{1-L/10}. \]
We recall that finding a copy of $({\bf a},{\bf e})_S$ in $G$ corresponds exactly to finding a witness, i.e. a collection of sets $\{A_{i,w}\}_{(i,w) \in S}$, such that for each $(i,w)\in S$ the set $A_{i,w}\subseteq V_i$ has cardinality $a_{i,w}$, and such that $e(A_{i,w},A_{i',w'})=e_{i,w,i',w'}$ whenever $(i,w)\sim_{\Gamma} (i',w')$.

Since there are at most
\[ \prod_{(i,w)\in S}\binom{n}{a_{i,w}\wedge \lfloor n/2\rfloor}\]
choices of the collection $\{A_{i,w}\}_{(i,w) \in S}$ it suffices to prove that
\[ 
\p{\bigcap_{(i,w)\sim_{\Gamma} (i',w')} \{e(A_{i,w},A_{i',w'})=e_{i,w,i',w'}\}}\ll \prod_{(i,w) \in S} \!\! a_{i,w}^{d/4}\, \pran{\prod_{(i,w) \in S} {n \choose a_{i,w}\wedge \lfloor n/2 \rfloor}}^{-L/10}
\] 
for each such collection $\{A_{i,w}\}_{(i,w) \in S}$.  Furthermore, since the inequality $(en/a)^a \ge {n \choose a\wedge \lfloor n/2 \rfloor}$ holds for all $1\le a\le n$, it suffices to prove that 
\begin{align*}
& \quad \p{\bigcap_{(i,w)\sim_{\Gamma} (i',w')} \{e(A_{i,w},A_{i',w'})=e_{i,w,i',w'}\}} \\
\ll & \quad \prod_{(i,w) \in S}\!\! a_{i,w}^{d/4}\, \cdot \, \exp\left(-\frac{L}{10}\, \cdot \sum_{(i,w) \in S} a_{i,w}\, \log\left(\frac{en}{a_{i,w}}\right)\right)
\end{align*}
for each such collection $\{A_{i,w}\}_{(i,w) \in S}$.

The required bound now follows by applying \refC{cor:bb} for each matching, then using the independence of the matchings and the second bound of \refP{prop:redbound}. \end{proof}

The rest of the section is dedicated to proving \refP{prop:redbound}.  We divide into two cases based on whether the potency of $({\bf a},{\bf e})$ derives mostly from large deviation terms (terms in which $\eps_{i,w,i',w'}$ is particularly large) or small deviation terms.  The reason for splitting the proof in this way is that the expression $b(\eps_{i,w,i',w'})$ behaves rather differently in the two cases, resembling $\eps\log{\eps}$ and $\eps^2$ respectively.

We shall partition $\Gamma$ into two subgraphs $\Gamma_{\LD}=\Gamma_{\LD}\pat$ and $\Gamma_{\SD}=\Gamma_{\SD}\pat$ by setting
\[ V(\Gamma_{\LD}) \, =\, V(\Gamma_{\SD}) \, = \, V(\Gamma)\]
and defining
\begin{align*}
E(\Gamma_{\LD}) & \, =\, \{(i,w)\sim_{\Gamma} (i',w')\, :\, \eps_{i,w,i',w'}\pat> e^2 -1\}\, ,
\\  
E(\Gamma_{\SD}) & \, =\, \{(i,w)\sim_{\Gamma} (i',w')\, :\, -1 \le \eps_{i,w,i',w'}\pat\le e^2 -1\} \, .
\end{align*}
We have chosen $e^2-1$ as the cutoff between the large and small deviations regimes as a matter of technical convenience to do with the details of the proof.
Note that the potency $p({\bf a},{\bf e})$ of a pattern $({\bf a},{\bf e})$ may be expressed as
\[ p({\bf a},{\bf e}) \, = \,  \left| \sum_{(i,w)\simg (i',w')}\frac{w w' a_{i,w} a_{i',w'}}{n} \, \eps_{i,w,i',w'}  \right|\, .\] 
We define now two variants of potency.
\[ p_{\LD}({\bf a},{\bf e})\, = \, \left|\sum_{(i,w)\sim_{\Gamma_{\LD}} (i',w')}\frac{w w' a_{i,w} a_{i',w'}}{n} \, \eps_{i,w,i',w'} \right|   \, .\] 
and
\[ p_{\SD}({\bf a},{\bf e})\, = \,  \left|\sum_{(i,w)\sim_{\Gamma_{\SD}} (i',w')}\frac{w w' a_{i,w} a_{i',w'}}{n} \, \eps_{i,w,i',w'}  \right|\, .\] 
The latter expressions respectively capture the ``large deviations potency'' and ``small deviations potency.'' 
Since $p({\bf a},{\bf e}) \le p_{\LD}({\bf a},{\bf e})+p_{\SD}({\bf a},{\bf e})$, for every pattern $({\bf a},{\bf e})$ we have that \[ \max\{p_{\LD}({\bf a},{\bf e}),p_{\SD}({\bf a},{\bf e})\}\ge \frac{p({\bf a},{\bf e})}{2}\, .\]  
We shall prove the following propositions.

\begin{prop}\label{prop:rbld} Fix $L \ge 20$.
For any pattern $({\bf a},{\bf e})$, there exists $S=S(({\bf a},{\bf e}),L) \subset V(\Gamma)$ such that the following properties hold: 
\[
p_{\LD}(({\bf a},{\bf e})_{S})\geq p_{\LD}({\bf a},{\bf e})-30L\sqrt{d} \,  ,
\]
and, for all $(i,w) \in S$, 
\[
\sum_{(i',w')\in N_{\Gamma_{\LD}}(i,w)\cap S} \frac{a_{i,w}a_{i',w'}}{n} \, \cdot \, \pran{1+\frac{\eps_{i,w,i',w'}}{2}} \log(1+\eps_{i,w,i',w'})  \ge  \frac{L}{4} a_{i,w}\log\left(\frac{en}{a_{i,w}}\right)
\] 
\end{prop}
\begin{prop}\label{prop:rbsd} Fix $L \ge 20$.
For any pattern $({\bf a},{\bf e})$, there exists $S=S(({\bf a},{\bf e}),L) \subset V(\Gamma)$ such that the following properties hold.
\begin{align*}
p_{\SD}(({\bf a},{\bf e})_{S}) 	& \geq p_{\SD}({\bf a},{\bf e})-55L\sqrt{d}\qquad \text{and}\\
\sum_{(i',w')\in N_{\Gamma_{\SD}}(i,w)\cap S} \frac{a_{i,w}a_{i',w'}}{n} \, \cdot \, \frac{\eps_{i,w,i',w'}^2}{15} & \ge  \frac{L}{10} a_{i,w}\log\left(\frac{en}{a_{i,w}}\right) \qquad \text{for all} \, (i,w)\in S\, .\end{align*}
\end{prop}

Since 
\begin{equation}\label{eq:b} 
b(\eps) \, \ge  \begin{cases}   \frac{\eps^2}{15}	&  \mbox{ if } \eps \leq e^2-1 \\
				\pran{1+\frac{\eps}{2}} \log(1+\eps) & \mbox{ if } \eps > e^2 -1\, ,
			\end{cases}
			\end{equation}
\refP{prop:redbound} follows immediately from Propositions~\ref{prop:rbld} and \ref{prop:rbsd} by applying the appropriate proposition to the pattern $({\bf a},{\bf e})$ (i.e. applying \refP{prop:rbld} if $p_{\LD}\pat \ge p\pat/2$ and applying \ref{prop:rbsd} if $p_{\SD}\pat \ge p\pat/2$). The required bound on $\tilde{p}(\pat_S)$ is obtained by considering the set $E\subset E(\Gamma)$ given by $E=\Gamma_{\LD}|_S$ or by $E=\Gamma_{\SD}|_S$, as appropriate.

The proofs of \refP{prop:rbld} and \refP{prop:rbsd} are similar.  In both cases the set $S$ is found by repeatedly applying some straightforward reduction rules. In spirit, these rules can be understood by analogy with the following procedure. 
Suppose we are given a graph $F=(V,E)$ with average degree $\mu$. To find an induced subgraph with {\em minimum} degree at least $\mu/2$, we may simply repeatedly throw away vertices of degree less than $\mu/2$ until no such vertices remain. Throwing away vertices of such small degree can only increase the average degree in what remains, so this procedure must terminate with a non-empty subgraph satisfying the desired global minimum degree requirement. In this example the measure of `importance' of a vertex is its degree. When we apply a similar style of argument, the analogue of degree will be ``local potency'', suitably defined. Also, our rules for throwing away vertices will be more involved, and so it will take some work to verify that in throwing away vertices we do not decrease the overall potency by too much. 


We now proceed to the proofs of Propositions~\ref{prop:rbld} and ~\ref{prop:rbsd}.   We prove \refP{prop:rbld} first since the reductions required for its proof are simpler.

\subsection{Large deviation patterns: A proof of \refP{prop:rbld}}\label{sec:ldps}
Now and for the remainder of Section~\ref{sec:ldps}, we fix $L\ge 20$ and a pattern $\pat$.
Our aim is to find $S \subset V(\Gamma)$ with $p_{\LD}(({\bf a},{\bf e})_{S})  \geq p_{\LD}({\bf a},{\bf e})-30L\sqrt{d}$ and such that the inequality 
\[ \sum_{(i',w')\in N_{\Gamma_{\LD}}(i,w)\cap S} \frac{a_{i,w}a_{i',w'}}{n} \, \cdot \, \pran{1+\frac{\eps}{2}} \log(1+\eps) \ge  \frac{L}{4} a_{i,w}\log\left(\frac{en}{a_{i,w}}\right) \]
holds for all $(i,w)\in S$.

As alluded to above, we shall choose $S$ by repeatedly removing vertices of $V(\Gamma)$ that correspond to ``inconsequential'' parts of the pattern, so that in the resulting sub-pattern $({\bf a},{\bf e})_S$ there is a significant contribution to the overall potency from every vertex.  We write  
\[ p_{\LD}(({\bf a},{\bf e});(i,w)) = \left|\sum_{(i',w')\in N_{\Gamma_{\LD}}(i,w)} \frac{w w' a_{i,w} a_{i',w'}}{n} \, \eps_{i,w,i',w'}\right| \, ,\]
for the `local' large deviations potency at $(i,w)$, 
and write 
\[ \widehat{N}_{i,w}\pat = \sum_{(i',w')\in N_{\Gamma}(i,w)} (w')^2 a_{i',w'} \left(\frac{w'}{wd^{1/2}}\right)\]
for the amount of ``$L_2$-squared weight'' of the type ${\bf a}$ that appears on $\Gamma$-neighbours of $(i,w)$ weighted by the factor $w'/wd^{1/2}$.

%

The intuition behind this weighting factor is that in the large deviations regime, the most significant contributions to potency should be made by edges $(i,w)(i',w') \in E(\Gamma_{\LD})$ with $w$ and $w'$ close to a factor of $d^{1/2}$ apart. Note that this is {\em exactly} what happens for the eigenfunctions of the universal cover (the infinite $d$-ary tree) with near-supremal eigenvalues.

We say that a set $U \subset V(\Gamma)$ satisfies condition (LD1) if for each $(i,w) \in U$ we have 
\begin{itemize}
\item[(LD1)] $p_{\LD}(({\bf a},{\bf e})_U;(i,w)) \ge L a_{i,w}w^2 d^{1/2}$,
\end{itemize}
and that $U$ satisfies condition (LD2) if for each $(i,w) \in U$ we have 
\begin{itemize}
\item[(LD2)] $p_{\LD}(({\bf a},{\bf e})_S;(i,w)) \ge L \widehat{N}_{i,w}\pat /d^{1/2}$.
\end{itemize}
Condition (LD1) asks that $p_{\LD}(({\bf a},{\bf e})_S;(i,w))$ is large relative to the size of $a_{i,w}$; the factor $a_{i,w}w^2$ approximately measures the proportion of $L^2$-squared weight of ${\bf a}$ used by $a_{i,w}$.  In Condition (LD2) the factor $a_{i,w}w^2$ is replaced by $\widehat{N}_{i,w}/d$ which (in some sense) measures the average level of opportunity available to $a_{i,w}$.  (In the case of large deviations a good opportunity corresponds to $(i',w') \in N_{\Gamma_{\LD}}(i,w) \cap U$ with large $a_{i',w'}$ large and such that the ratio $w'/w$ is close to $d^{1/2}$.) 

We then let $S=S(({\bf a},{\bf e}),L) \subset V(\Gamma)$ be the maximal subset of $V(\Gamma)$ satisfying both (LD1) and (LD2). 
The monotonicity of $p_{\LD}(({\bf a},{\bf e})_U;(i,w))$ in $U$ guarantees that $S$ is unique and that $S$ can be determined by starting from $V(\Gamma)$ by repeatedly throwing away vertices that violate at least one of the two conditions, until no such vertices remain. 
In other words, writing $k = |V(\Gamma)|-|S|$, 
we may order the vertices of $V(\Gamma) \setminus S$ as $(i_1,w_1),\ldots,(i_k,w_k)$ in such a way that 
for each $1 \leq j \leq k$, letting $S_j = V(\Gamma) \setminus \{(i_1,w_1),\ldots,(i_{j-1},w_{j-1})\}$, 
we have 
\[
p_{\LD}(({\bf a},{\bf e})_{S_j};(i_j,w_j))<L \max\left( a_{i,w}w^2 d^{1/2}, \widehat{N}_{i,w} /d^{1/2}\right)\]
In other words, this ordering ``verifies'' that $(i_j,w_j) \not \in S$, for each $1 \leq j \leq k$. 

Proving \refP{prop:rbld} now reduces to establishing the following two lemmas.

\begin{lem}\label{lem:potboundLD} Let $S=S(({\bf a},{\bf e}),L) \subset V(\Gamma)$ be as defined above.  Then $p_{\LD}(({\bf a},{\bf e})_{S})  \geq p_{\LD}({\bf a},{\bf e})-30L\sqrt{d}$.\end{lem}

\begin{lem}\label{lem:suffLD} 
Given any subset $U$ of $V(\Gamma)$, if $U$ satisfies (LD1) and (LD2) then for all $(i,w) \in U$, 
\[
\sum_{(i',w')\in N_{\Gamma}(i,w)\cap U} \frac{a_{i,w}a_{i',w'}}{n} \, \cdot \, \pran{1+\frac{\eps_{i,w,i',w'}}{2}} \log(1+\eps_{i,w,i',w'}) \ge  \frac{L}{4} a_{i,w}\log\left(\frac{en}{a_{i,w}}\right)\, .
\]
\end{lem}

\begin{proof}[Proof of \refL{lem:potboundLD}.]  With $(i_1,w_1),\ldots,(i_k,w_k)$ and $S_1,\ldots,S_k$ as above, 
write
\[
W_1 = \{ (i_j,w_j): 1 \le j \le k, p_{\LD}(({\bf a},{\bf e})_{S_j};(i_j,w_j))< L a_{i,w}w^2 d^{1/2}\}
\]
and write 
\[
W_2= \{ (i_j,w_j), 1 \le j \le k, p_{\LD}(({\bf a},{\bf e})_{S_j};(i_j,w_j))< L \widehat{N}_{i,w} /d^{1/2}\}\, .
\]
We note that $W_1 \cup W_2 =\{(i_1,w_1),\dots ,(i_k,w_k)\}$.
Since
\[
p_{\LD}(\pat_S) \ge p_{\LD}(\pat) - \sum_{1 \leq j \leq k} p_{\LD}(({\bf a},{\bf e})_{S_j};(i_j,w_j)),
\]
to prove the proposition it suffices to show that 
\begin{equation}\label{thirty}
\sum_{1 \leq j \leq k} p_{\LD}(({\bf a},{\bf e})_{S_j};(i_j,w_j)) \leq 30Ld^{1/2}. 
\end{equation}
First, for any $1 \leq j \leq k$ with $(i_j,w_j) \in W_1$, 
we have 
$p_{\LD}(({\bf a},{\bf e})_{S_j};(i_j,w_j))\le L a_{i,w}w^2 d^{1/2}$, and so 
\[
\mathop{\sum_{1 \leq j \leq k}}_{(i_j,w_j) \in W_1}p_{\LD}(({\bf a},{\bf e})_{S_j};(i_j,w_j)) 
\leq 
 \mathop{\sum_{1 \leq j \leq k}}_{(i_j,w_j) \in W_1} L a_{i,w} w^2 d^{1/2}\le 10 Ld^{1/2}\, ,\] 
the last inequality holding since ${\bf a}$ is a Z-type so $\sum_{(i',w') \in V(\Gamma)} w'^2 a_{i',w'} \le 10$. 
 
Now, for $(i_j,w_j)\in W_2$, we have that $p_{\LD}(({\bf a},{\bf e})_{S_j};(i_j,w_j))\le L \widehat{N}_{i,w}/d^{1/2}$, where (we recall) 
\[
\widehat{N}_{i,w}=\sum_{(i',w')\in N_{\Gamma}(i,w)}a_{i',w'}w'^2\left(\frac{w'}{wd^{1/2}}\right)\] 
Note that for any $(i,w) \in V(\Gamma)$ and any neighbour $(i',w')$ of $(i,w)$ we have 
$w'/wd^{1/2} \leq 1$. Since all weights $w$ are in $\dog$, it follows that  
\begin{equation}\label{twod}
\sum_{(i,w)\in N_{\Gamma}(i',w')} \frac{w'}{wd^{1/2}}
= \sum_{i \in N_H(i')} \sum_{w \in \dog, w \le d^{1/2} w'} \frac{w'}{wd^{1/2}}
\le \sum_{i \in N_H(i')} \sum_{j \in \N} 2^{-j}
= 2d
\end{equation}
for each $(i',w')\in V(\Gamma)$.  Thus
\begin{align*}
\sum_{(i,w) \in W_2} p_{\LD}(({\bf a},{\bf e})_{S_j};(i_j,w_j)) & 
\le \frac{L}{d^{1/2}}\, \cdot \, \sum_{(i,w)\in W_2} \sum_{(i',w')\in N_{\Gamma}(i,w)}a_{i',w'}w'^2\left(\frac{w'}{wd^{1/2}}\right) \\
&
\le \frac{L}{d^{1/2}}\, \cdot \, \sum_{(i',w')\in V(\Gamma)} a_{i',w'}w'^2\, \cdot \sum_{(i,w)\in N_{\Gamma}(i',w')} \frac{w'}{wd^{1/2}}\\
&
\le 20 Ld^{1/2}\, ,\end{align*}
where the final inequality follows from \eqref{twod} and the fact that ${\bf a}$ is a $Z$-type.

Since $W_1\cup W_2=\{(i_1,w_1),\dots ,(i_k,w_k)\}$, the above bounds establish \eqref{thirty}, completing the proof of the proposition.\end{proof}

In preparation for the proof of \refL{lem:suffLD}, we first record the following fact. 
Given $U \subset V(\Gamma)$ and $(i,w) \in V(\Gamma)$, define 
\[
U_{\LD}(i,w) = \left\{(i',w') \in N_{\Gamma_{\LD}}(i,w)\cap U: \frac{\eps_{i,w,i',w'}w^2 d}{w'^2}\,\ge \, \frac{Ln}{2a}\right\} \, .
\]
The set $U_{\LD}(i,w)$ contains the vertices we shall view as making an important contribution in our forthcoming use of \refL{lem:measure}.
\begin{lem}\label{lem:pointwise} If $U \subset V(\Gamma)$ satisfies (LD2) then 
for all $(i,w)\in U$, 
\[
\sum_{(i',w')\in U_{\LD}(i,w)} \frac{w w' a_{i,w} a_{i',w'}}{n} \, \eps_{i,w,i',w'} \ge \frac{1}{2} p_{\LD}(({\bf a},{\bf e})_U;(i,w))\, 
\]
\end{lem}
\begin{proof}
Fix $(i,w) \in U$, let 
$\Omega = N_{\Gamma_{\LD}}(i,w)\cap U$, and let $\cE = 2^{\Omega}$. 
Then for $(i',w') \in \Omega$, set 
\[
\mu(i',w') = a_{i,w}a_{i',w'}, \quad h(i',w') = \frac{ww'\eps_{i,w,i',w'}}{n}, \quad g(i',w')=\frac{w'^2}{nd^{1/2}}\frac{w'}{wd^{1/2}}. 
\]
Using (LD2), we then have 
\[
\int h d\mu =  p_{\LD}(({\bf a},{\bf e})_U;(i,w))\ge L\frac{\widehat{N}_{i,w}}{d^{1/2}}
\]
and
\[ 
\int g\, d\mu = \frac{a_{i,w}}{nd^{1/2}} \sum_{(i',w')\in N_{\Gamma_{\LD}}(i,w)\cap U}  a_{i',w'}w'^2\, \frac{w'}{wd^{1/2}} \ge\frac{a_{i,w}\widehat{N}_{i,w}}{nd^{1/2}}\, .
\]
It follows that 
\[
U_{\LD}(i,w) \supseteq \left\{(i',w') \in N_{\Gamma_{\LD}}(i,w)\cap U: \frac{h(i',w')}{g(i',w')} \geq \frac{1}{2}\frac{\int h d\mu}{\int g d\mu}\right\},
\]
and so by \refL{lem:measure} applied with $c=1/2$, 
\[
\sum_{(i',w')\in U_{\LD}(i,w)} \frac{w w' a_{i,w} a_{i',w'}}{n} \ge \int_{\frac{h}{g} \geq \frac{1}{2}\frac{\int h}{\int g}} h \, d\mu \ge 
\frac{1}{2} \int_{\Omega} h \, d\mu = \frac{1}{2} p_{\LD}(({\bf a},{\bf e})_U;(i,w)).
\]
\end{proof}

We now prove \refL{lem:suffLD}, completing the proof of \refP{prop:rbld}.

\begin{proof}[Proof of \refL{lem:suffLD}] 
We are given $U \subset V(\Gamma)$ satisfying (LD1) and (LD2), and aim to show that 
\[
\sum_{(i',w')\in N_{\Gamma}(i,w)\cap U} \frac{a_{i,w}a_{i',w'}}{n} \, \cdot \, \pran{1+\frac{\eps_{i,w,i',w'}}{2}} \log(1+\eps_{i,w,i',w'}) \ge  \frac{L}{4} a_{i,w}\log\left(\frac{en}{a_{i,w}}\right)\, .
\]

Note that $c\log{x}\ge \log(c^2 x)$ for $x\ge e^2$ and $c\ge 1$.  Since $1+\eps_{i,w,i',w'}\ge e^2$ for all $(i',w')\in N_{\Gamma_{\LD}}(i,w)$ we have that 
\[
\frac{w d^{1/2}}{w'} \log(1+\eps_{i,w,i',w'}) \ge \log\left((1+\eps_{i,w,i',w'})\left(\frac{w^2 d}{(w')^2}\right)\right)\ge \log\left(\frac{\eps_{i,w,i',w'}w^2 d}{(w')^2}\right)\, .
\]
It follows that
\begin{align}
& \quad \frac{a_{i,w}a_{i',w'}}{n} \, \cdot \, \pran{1+\frac{\eps_{i,w,i',w'}}{2}} \log(1+\eps_{i,w,i',w'})  \nonumber\\
= & \quad 
\frac{1}{w^2 d^{1/2}} \frac{ww'a_{i,w}a_{i',w'}}{n} \, \cdot \left(\frac{wd^{1/2}}{w'}\right)\, \cdot \pran{1+\frac{\eps_{i,w,i',w'}}{2}} \log(1+\eps_{i,w,i',w'}) \nonumber\\
\ge & \quad 
\frac{1}{2w^2 d^{1/2}}\frac{ww'a_{i,w}a_{i',w'}}{n} \eps_{i,w,i',w'} \log\left(\frac{\eps_{i,w,i',w'}w^2 d}{(w')^2}\right)\, \label{eq:rbld1}.
\end{align}

Note that the expression inside the preceding logarithm is precisely the expression included in the definition of $U_{\LD}(i,w)$.  
Applying first (\ref{eq:rbld1}), then \refL{lem:pointwise} and finally (LD1), we obtain that 
\begin{align*}
& \quad \sum_{(i',w')\in N_{\Gamma_{\LD}}(i,w)\cap U} \frac{a_{i,w}a_{i',w'}}{n} \, \cdot \, \pran{1+\frac{\eps_{i,w,i',w'}}{2}} \log(1+\eps_{i,w,i',w'}) \\
\ge & \quad
\frac{1}{2w^2 d^{1/2}}\sum_{(i',w')\in N_{\Gamma_{\LD}}(i,w)\cap U} \frac{ww'a_{i,w}a_{i',w'}}{n} \eps_{i,w,i',w'} \log\left(\frac{\eps_{i,w,i',w'}w^2 d}{w'^2}\right)\\
\ge & \quad 
\frac{1}{4w^2 d^{1/2}}p_{\LD}(({\bf a},{\bf e})_U;(i,w))\log\left(\frac{Ln}{2a}\right)\\
\ge & \quad 
\frac{L}{4} a_{i,w}\log\left(\frac{en}{a_{i,w}}\right)\, ,
\end{align*}
completing the proof of the lemma.
\end{proof}

\subsection{Small deviation patterns: A proof of \refP{prop:rbsd}}\label{sec:rbsd}
Our proof of \refP{prop:rbsd} is similar to our proof of \refP{prop:rbld} given above.  In that spirit, throughout Section~\ref{sec:rbsd} we fix $L\ge 20$ and a pattern $\pat$. We will find the required set $S$ by repeatedly removing vertices of $V(\Gamma)$ that make a small contribution to potency.

Given a pattern $\pat$ and $(i,w) \in V(\Gamma)$, we write 
\[ p_{\SD}(({\bf a},{\bf e});(i,w)) = \left|\sum_{(i',w')\in N_{\Gamma_{\SD}}(i,w)} \frac{w w' a_{i,w} a_{i',w'}}{n} \, \eps_{i,w,i',w'}\right| \, ,\]
for the `local' small deviations potency at $(i,w)$, 
and write 
\[
N_i = N_i \pat = \sum_{(i',w') \in V(\Gamma): i' \sim_H i} w'^2 a_{i',w'}.
\]
for the proportion of the $L_2$-squared weight of ${\bf a}$ that appears on fibres $V_{i'}$ that neighbour $V_i$. 
Our next aim is to state small deviations analogues of Lemmas~\ref{lem:potboundLD} and~\ref{lem:suffLD}. However, our reduction rules are slightly more involved in this case, and require one additional definition. 
Let 
\[
M_{i,w} = M_{i,w}\pat = \max\left\{\frac{N_i}{a_{i,w}w^2 d},\frac{en}{a_{i,w}}\right\}\,, 
\]
and let $m_{i,w} = (\log M_{i,w})/M_{i,w}$. Note that we always have $M_{i,w} \ge en/a_{i,w} \ge e$. Since the function $x \log x^{-1}$ is increasing on $(0,e^{-1}]$ it follows that we may equivalently write 
\[
m_{i,w} = 	\begin{cases}
			\frac{a_{i,w} w^2 d}{N_i} \log\left(\frac{N_i}{a_{i,w} w^2 d}\right) & \mbox{ if } 
			\frac{w^2 d}{N_i} \le \frac{1}{en} \\
			\frac{a_{i,w}}{en} \log\left(\frac{en}{a_{i,w}}\right) & \mbox{ if }
			\frac{1}{en} \le \frac{w^2 d}{N_i} 
			\end{cases}
\]
In what follows we will use that $m_{i,w} \le 1.18 M_{i,w}^{-2/3}$ always holds. 
This follows from the fact that $\log x \le 1.18 x^{1/3}$ on $(0,\infty)$.
 
We say that a set $U \subset V(\Gamma)$ satisfies (SD1), (SD2), and (SD3), respectively, if
\begin{itemize}
\item[(SD1)] $p_{\SD}(({\bf a},{\bf e})_U;(i,w)) \ge L a_{i,w}w^2 d^{1/2}$
\item[(SD2)] $p_{\SD}(({\bf a},{\bf e})_U;(i,w)) \ge L N_i\pat  a_{i,w}/(nd^{1/2})$, or 
\item[(SD3)] $p_{\SD}(({\bf a},{\bf e})_U;(i,w)) \ge L N_i\pat m_{i,w}\pat/d^{1/2} \, $
\end{itemize}
for all $(i,w)\in U$.

We remark that (SD1) is identical to (LD1), and asks that $p_{\SD}(({\bf a},{\bf e})_U;(i,w))$ is large relative to $a_{i,w}w^2$.
In (SD2), $a_{i,w}/n$ is the proportion of the fibre $V_i$ consumed by a set of size $a_{i,w}$, 
and $N_i / d$ represents the `average opportunity' in the neighbourhood of $(i,w)$. 
Notice that unlike in the large deviations case, no factor of the form $(w'/wd^{1/2})$ appears. This corresponds to the intuition that in the small deviations case, large potency is most likely to come from large sets of roughly equal weight. 
Finally, (SD3), in which $m_{i,w}$ appears, is a slight strengthening of either (SD1) or (SD2) -- by a logarithmic factor -- depending on which value $m_{i,w}$ takes. 

Let $S=S(\pat,L)$ be the maximal subset of $V(\Gamma)$ such that every $(i,w) \in S$ satisfies (SD1), (SD2), and (SD3). As in \refS{sec:ldps}, writing $k = |V(\Gamma)|-|S|$, 
we may order the vertices of $V(\Gamma) \setminus S$ as $(i_1,w_1),\ldots,(i_k,w_k)$ in such a way that 
for each $1 \leq j \leq k$, letting $S_j = V(\Gamma) \setminus \{(i_1,w_1),\ldots,(i_{j-1},w_{j-1})\}$, 
we have 
\[
p_{\SD}(({\bf a},{\bf e})_{S_j};(i_j,w_j))<L \max\left( a_{i,w}w^2 d^{1/2}, N_{i}a_{i,w} /(nd^{1/2}),N_i m_{i,w}/d^{1/2})\right)\]
In other words, this ordering verifies that $(i_j,w_j) \not \in S$, for each $1 \leq j \leq k$. 
\refP{prop:rbsd} is an immediate consequence of the following two lemmas. 
\begin{lem}\label{lem:potboundSD}
Let $S=S(\pat,L)$ be as defined above. Then we have $p_{\SD}(\pat_S) \ge p_{\SD}(\pat)-55L\sqrt{d}$. 
\end{lem}
\begin{lem}\label{lem:suffSD}
Given any subset $U$ of $V(\Gamma)$, if $U$ satisfies (SD1), (SD2), and (SD3) then for all $(i,w) \in U$, 
\[
\sum_{(i',w')\in N_{\Gamma_{\SD}}(i,w)\cap U} \frac{a_{i,w}a_{i',w'}}{n} \, \cdot \, \frac{\eps_{i,w,i',w'}^2}{15}  \ge  \frac{L}{10} a_{i,w}\log\left(\frac{en}{a_{i,w}}\right)\, .
\]
\end{lem}
\begin{proof}[Proof of \refL{lem:potboundSD}]
Write 
\begin{align*}
W_1 & = \{ (i_j,w_j): 1 \le j \le k, p_{\SD}(({\bf a},{\bf e})_{S_j};(i_j,w_j))< L a_{i,w}w^2 d^{1/2}\} \\
W_2  & = \{ (i_j,w_j), 1 \le j \le k, p_{\SD}(({\bf a},{\bf e})_{S_j};(i_j,w_j))< L N_{i}a_{i,w} /(nd^{1/2})\} \\
W_3  & = \{ (i_j,w_j), 1 \le j \le k, p_{\SD}(({\bf a},{\bf e})_{S_j};(i_j,w_j))< L N_{i}m_{i,w} /d^{1/2}\}\, .
\end{align*}
Also, let $W_3^-= W_3 \setminus W_1$. We note that for $(i,w) \in W_3^{-}$ we have $m_{i,w} \ge a_{i,w} w^2 d/N_i$.
Since $M_{i,w}\ge e$ in all cases, we have that $m_{i,w}\le e^{-1}$ and so
\begin{equation}\label{eq:forclaim}
\frac{a_{i,w} w^2 d}{N_i}  \le e^{-1},
\end{equation}
a bound we will use later.

Since $W_1 \cup W_2 \cup W_{3}^{-} = \{(i_1,w_1),\dots ,(i_k,w_k)\}$ 
and 
\[
p_{\SD}(\pat_S) \ge p_{\SD}(\pat) - \sum_{1 \le j \le k} p_{\SD}(({\bf a},{\bf e})_{S_j};(i_j,w_j))\, , 
\]
it suffices to prove that 
\[ 
\sum_{(i_j,w_j) \in W_1 \cup W_2 \cup W_{3}^{-}} p_{\SD}(({\bf a},{\bf e})_{S_j};(i_j,w_j))\le 55Ld^{1/2}\, .
\]
The proof that 
\begin{equation}\label{eq:wone}
\sum_{(i,w) \in W_1} p_{\SD}(({\bf a},{\bf e})_{S_j};(i_j,w_j))\le 10 Ld^{1/2}\,
\end{equation}
is identical to that given in the proof of \refL{lem:potboundLD}.  

For $W_2$ we use the fact that for all $i' \in V(H)$, 
\[
\sum_{(i,w): i\in N_H (i')} a_{i,w}\le nd\, ,
\]
and proceed as follows.
\begin{align}
\sum_{(i_j,w_j) \in W_2} p_{\SD}(({\bf a},{\bf e})_{S_j};(i_j,w_j))
&
\le 
\frac{L}{nd^{1/2}} \sum_{(i_j,w_j) \in W_2} a_{i_j,w_j} \sum_{(i',w'):i'\in N_H (i_j)}w'^2 a_{i',w'}
\nonumber\\
& 
\le
\frac{L}{nd^{1/2}} \sum_{(i',w')\in V(\Gamma)} w'^2 a_{i',w'} \sum_{(i,w): i\in N_H (i')} a_{i,w}
\nonumber\\
&
=
10 L d^{1/2},\label{eq:wtwo}
\end{align}
the last inequality holding since ${\bf a}$ is a Z-type so $\sum_{(i',w') \in V(\Gamma)} w'^2 a_{i',w'} \le 10$. 

The argument for $W_3^-$ is slightly more involved.  
We shall deduce our bound on $\sum_{(i_j,w_j) \in W_3^-} p_{\SD}(({\bf a},{\bf e})_{S_j};(i_j,w_j))$ from the following claim.

\textbf{Claim:} For each $i\in V(H)$ and each $j\ge 1$ there are at most $(j-1)$ vertices $(i,w)\in W_3^-$ for which $M_{i,w}\le 2^j$.

\textbf{Proof.}
Fix $j \ge 1$. First, by (\ref{eq:forclaim}) we have that $w^2 \le N_i/(ea_{i,w}d)$. If 
$M_{i,w} \le 2^j$ then by the definition of $M_{i,w}$ we also have 
$a_{i,w} \ge en/2^j$, and so 
\[
w^2 \le \frac{2^jN_i}{e^2 nd} < \frac{2^{j-2}N_i}{nd}.
\]
Using the other term in the maximum which defines $M_{i,w}$, we see that if $M_{i,w} \le 2^j$ then we also have 
\[
w^2 \ge \frac{N_i}{2^ja_{i,w}d} \ge \frac{N_i}{2^j n d}.
\] 
Since $w$ is a dyadic multiple of $(nd)^{-1/2}$ and our upper and lower bounds for $w^2$ are a factor of less than $2^{2j-2}$ apart, the claim follows. \qed

By the above claim and the fact, noted earlier, that $m_{i,w} \le 1.18 M_{i,w}^{-2/3}$ always holds, we have 
\begin{equation}\label{nine}
\sum_{w:(i,w)\in W_3^-} m_{i,w} \le \sum_{w:(i,w)\in W_3^-} 1.18 M_{i,w}^{-2/3} \le \sum_{j \ge 1} \frac{(j-1)}{2^{2j/3}}< 3.5.
\end{equation}
We are now ready to complete the proof of the lemma.  We have 
\begin{align*}
\sum_{(i_j,w_j) \in W_3^-} p_{\SD}(({\bf a},{\bf e})_{S_j};(i_j,w_j))
&
\le
\frac{L}{d^{1/2}} \sum_{(i_j,w_j) \in W_3^-}  m_{i_j,w_j} \sum_{(i',w'):i'\in N_H (i_j)}w'^2 a_{i',w'}\\
&
\le
\frac{L}{d^{1/2}} \sum_{(i',w')\in V(\Gamma)} w'^2 a_{i',w'} \sum_{(i,w)\in W_3^-: i\in N_H (i')} m_{i,w}\\
&
\le
\frac{L}{d^{1/2}} \sum_{(i',w')\in V(\Gamma)} w'^2 a_{i',w'} \sum_{(i,w)\in W_3^-: i\in N_H (i')} 1.18M_{i,w}^{-2/3}\\
& 
< 
35Ld^{1/2}\, ,
\end{align*}
the first inequality following from the definition of $W_3$, and the final inequality holding from (\ref{nine}) and since ${\bf a}$ is a Z-type. Combining this bound with (\ref{eq:wone}) and (\ref{eq:wtwo}) completes the proof.
\end{proof}

The last step before proving \refL{lem:suffSD} and thereby completing the proof of \refP{prop:redbound}, is to establish an analogue of \refL{lem:pointwise} for the small deviations case. In this case it turns out to be more straightforward to argue directly rather than to use \refL{lem:measure}.

Given $U \subset V(\Gamma)$ and $(i,w) \in V(\Gamma)$, let 
\[
U_{\SD}(i,w)=\left\{(i',w')\in N_{\Gamma_{\SD}}(i,w) \cap U:\frac{|\eps_{i,w,i',w'}|}{ww'n}\ge \frac{p_{\SD}(({\bf a},{\bf e})_U;(i,w))}{2w^2 a_{i,w} N_i}\right\}\, .
\]
\begin{lem}\label{easylemma} 
For all $(i,w) \in U$, we have 
\[
\left|\sum_{(i',w')\in U_{\SD}(i,w)} \frac{w w' a_{i,w} a_{i',w'}}{n} \, \eps_{i,w,i',w'}\right| \ge \frac{1}{2} p_{\SD}(({\bf a},{\bf e})_U;(i,w))\, .
\]
\end{lem}

\begin{proof} Fix $(i,w) \in U$. Since 
\[
p_{\SD}(({\bf a},{\bf e})_U;(i,w)) 
= \left|\sum_{(i',w') \in N_{\Gamma_{\SD}}(i,w) \cap U}
\frac{ww'a_{i,w}a_{i',w'}}{n}\eps_{i,w,i',w'}\right|, 
\]
by the triangle inequality it suffices to show that 
\[
\sum_{(i',w')\in \overline{U}_{\SD}(i,w)} \frac{w w' a_{i,w} a_{i',w'}}{n} \, |\eps_{i,w,i',w'}|  \le
\frac{p_{\SD}(({\bf a},{\bf e})_U;(i,w))}{2}\, ,
\]
where $\overline{U}_{\SD}(i,w)$ denotes the set $(N_{\Gamma_{\SD}}(i,w)\cap U)\setminus U_{\SD}(i,w)$.
We observe that the above sum may be re-expressed as
\[
w^2 a_{i,w} \sum_{(i',w')\in \overline{U}_{\SD}(i,w)} w'^2 a_{i',w'}\, \left(\frac{|\eps_{i,w,i',w'}|}{ww'n} \right)\, ,
\]
and that the bound 
\[
\frac{|\eps_{i,w,i',w'}|}{ww'n}\le \frac{p_{\SD}(({\bf a},{\bf e})_U;(i,w))}{2w^2 a_{i,w} N_i}
\]
holds for all $(i',w')\in \overline{U}_{\SD}(i,w)$.  The required bound now follows immediately by noting that
\[
\sum_{(i',w')\in \overline{U}_{\SD}(i,w)} w'^2 a_{i',w'}\le \sum_{(i',w')\in N_{\Gamma_{\SD}}(i,w)}w'^2 a_{i',w'}\le N_i\, ,
\]
the last inequality holding by the definition of $N_i$. 
\end{proof}
%
\begin{proof}[Proof of \refL{lem:suffSD}] 
Fix $U \subset V(\Gamma)$ satisfying (SD1), (SD2), and (SD3), and fix $(i,w) \in U$. To prove the lemma, we must show that 
\[
\sum_{(i',w')\in N_{\Gamma_{\SD}}(i,w)\cap U} \frac{a_{i,w}a_{i',w'}}{n} \, \cdot \, \frac{\eps_{i,w,i',w'}^2}{15}  \ge  \frac{L}{10} a_{i,w}\log\left(\frac{en}{a_{i,w}}\right).
\]
The proof divides into two cases depending on which value $m_{i,w}$ takes.  Before considering these cases separately, we note the following inequality: for all $(i,w) \in S$, 
\begin{align*}
& \quad \sum_{(i',w')\in N_{\Gamma_{\SD}}(i,w)\cap U} \frac{a_{i,w}a_{i',w'}}{n} \, \cdot \, \frac{\eps_{i,w,i',w'}^2}{15} \\
= & \quad 
\frac{n}{15}\sum_{(i',w')\in N_{\Gamma_{\SD}}(i,w)\cap U} \frac{ww'a_{i,w}a_{i',w'}|\eps_{i,w,i',w'}|}{n} \, \cdot \, \left(\frac{|\eps_{i,w,i',w'}|}{ww'n}\right)\\
\ge & \quad 
\frac{n}{15}\sum_{(i',w')\in U_{\SD}(i,w)} \frac{ww'a_{i,w}a_{i',w'}|\eps_{i,w,i',w'}|}{n} \, \cdot \, \left(\frac{|\eps_{i,w,i',w'}|}{ww'n}\right)\\ 
\ge & \quad 
\frac{n}{15} \frac{p_{\SD}(({\bf a},{\bf e})_U;(i,w))}{2w^2 a_{i,w} N_i}\, \cdot 
\sum_{(i',w')\in U_{\SD}(i,w)} \frac{ww'a_{i,w}a_{i',w'}|\eps_{i,w,i',w'}|}{n}\\
\ge & \quad 
\frac{n\, p_{\SD}(({\bf a},{\bf e})_U;(i,w))^2}{60 w^2 a_{i,w} N_i}\, .
\end{align*}
the second inequality holding by the definition of $U_{\SD}(i,w)$ and the third holding by \refL{easylemma}. 
It therefore suffices to prove that for all $(i,w) \in S$, 
\begin{equation}\label{STP}  
\frac{n\, p_{\SD}(({\bf a},{\bf e})_U;(i,w))^2}{60 w^2 a_{i,w} N_i} \ge  \frac{L}{10} a_{i,w}\log\left(\frac{en}{a_{i,w}}\right)\, .
\end{equation}

First, if $w^2 d/N_i \le 1/en$, then $m_{i,w}=\frac{w^2 a_{i,w} d}{N_i}\log\left(\frac{N_i}{a_{i,w} w^2 d}\right)$. In this case, by (SD3), we have that
\[  p_{\SD}(({\bf a},{\bf e})_U;(i,w))\ge L w^2 a_{i,w} d^{1/2} \log\left(\frac{N_i}{w^2 a_{i,w} d}\right)\ge L w^2 a_{i,w} d^{1/2} \log\left(\frac{en}{a_{i,w}}\right)\, ,
\] 
where the final inequality follows since $a_{i,w} w^2 d/N_i \le a_{i,w}/en$.  Combining this bound with (SD2) yields that 
\[  p_{\SD}(({\bf a},{\bf e})_U;(i,w))^2  
\ge \frac{L^2 w^2 a_{i,w}^2N_i}{n} \log\left(\frac{en}{a_{i,w}}\right)\, ,
\] 
which verifies \eqref{STP} in this case, since $L \ge 20 > 6$.

In the remaining case we have $1/en \le w^2 d/N_i$, from which it follows that $m_{i,w}=a_{i,w}\log(en/a_{i,w})/en$. So, by (SD3), we have that
\[  
p_{\SD}(({\bf a},{\bf e})_U;(i,w))\ge \frac{La_{i,w}N_i}{end^{1/2}}\log(en/a_{i,w})\, .
\]
Combining this bound with (SD1), 
we obtain that 
\[
p_{\SD}(({\bf a},{\bf e})_U;(i,w))^2 \ge \frac{L^2a_{i,w}^2 w^2 N_i}{en}\log(en/a_{i,w}),
\]
which verifies \eqref{STP} in this case since $L\ge 20\ge 6e$.This completes the proof. 
\end{proof}

\section{Proof of \refT{explain}}\label{sec:explain}

Our proof of \refT{explain} is similar in nature to our proof of \refT{main}.  As in that proof we shall use \refP{findwit}, which tells us that if $\lambda^*(G)$ is large then $G$ contains a pattern $\pat$ of large potency.  We require two other results.  The first, which is easily proved, is an approximate converse to \refP{findwit}.  The second, whose proof is more involved, is a variant of \refT{redbound}.

Given a pattern $\pat$ and a witness $\{A_{i,w}\}_{(i,w) \in V(\Gamma)}$ for the pattern, write $\alpha=\alpha\pat=\sum_{(i,w) \in V(\Gamma)} a_{i,w}$, and write $G[\{A_{i,w}\}_{(i,w) \in V(\Gamma)}]$ for the subgraph of $G$ induced by $\bigcup_{(i,w) \in V(\Gamma)} A_{i,w}$. 
\begin{prop}\label{witfind}
If $G$ contains a pattern $\pat$ with $p\pat=p\sqrt{d}$ then $\lambda^*(G) \geq (2p-40)\sqrt{d}$. 
Furthermore, $\lambda(G[\{A_{i,w}\}_{(i,w) \in V(\Gamma)}]) \geq (2p-40)\sqrt{d} - \frac{\alpha\sqrt{10}}{n}$. 
\end{prop}
\begin{proof}
Let $\{A_{i,w}: (i,w) \in V(\Gamma)\}$ be a witness for $\pat$, 
and let $y \in \mathbb{R}^{nh}$ be defined by, for each 
$(i,w) \in V(\Gamma)$, setting $y(v)=w$ for all $v \in A_{i,w}$ (and setting $y(v)=0$ for all remaining $v$). 
Since 
\[
\|y\|_2^2 = \sum_{(i,w) \in V(\Gamma)} w^2 a_{i,w} \le 10, 
\]

by \refL{ipoe} we have 
\[
|\ipo{y}_N| = |\ipo{y}_{N,E^*} + \ipo{y}_{N,\R^2\setminus E^*}| \geq |\ipo{y}_{N,\R^2\setminus E^*}| - 40\sqrt{d}.
\]
But $2p\pat = |\ipo{y}_{N,\R^2\setminus E^*}|$, so $|\ipo{y}_N| \geq (2p-40)\sqrt{d}$, and the first result follows from \refF{uncount}. 

For the second bound, write $G'=G[\{A_{i,w}\}_{(i,w) \in V(\Gamma)}]$, and write $M'$ for the adjacency matrix of $G'$. 
Also, let $y' \in \R^{a}$ be the vector obtained from $y$ by retaining only coordinates corresponding to vertices of $G'$ 
(recall that all other coordinates of $y$ are equal to zero). Then 
\[
|\ipo{y}_N| = |\ipo{y}_M- \ipo{y}_{\overline{M}}| = |\ipo{y'}_{M'}- \ipo{y}_{\overline{M}}|,
\]
and since all entries of $\overline{M}$ are either zero or $1/n$, 
\[
\ipo{y}_{\overline{M}} \leq \frac{1}{n} (\sum_{v \in [nh]} y_v)^2 \leq \frac{\alpha\sqrt{10}}{n}, 
\]
where in the last inequality we used that, given the constraints that $y$ has $\alpha$ non-zero entries and $\|y\|_2^2 \leq 10$, the sum $(\sum_{v \in [nh]} y_v)^2$ is maximized by taking all nonzero entries equal to $10^{1/2}\alpha^{-1/2}$. 
Thus 
\[
\frac{|\ipo{y'}_{M'}|}{\|y'\|_2^2} \geq \ipo{y'}_{M'} \ge |\ipo{y}_N|-|\ipo{y'}_M| \geq (2p-40)\sqrt{d} - \frac{\alpha\sqrt{10}}{n}.\qedhere
\]
\end{proof}

\begin{thm}\label{redboundtwo}
Fix $L \ge 20$. 
For any pattern $({\bf a},{\bf e})$, there exists $S=S(({\bf a},{\bf e}),L) \subset V(\Gamma)$ such that the following properties hold.
\begin{align*}
p(({\bf a},{\bf e})_{S}) 	& \geq p({\bf a},{\bf e})- 150 L\sqrt{d}\\
\p{({\bf a},{\bf e})_S~\mbox{can be found in G}} & \ll \prod_{(i,w) \in S} a_{i,w}^{d/4}\pran{\prod_{(i,w) \in S} {n \choose a_{i,w}\wedge \lfloor n/2 \rfloor}}^{1-L/10}. 
\end{align*}
\end{thm}

In the same way as \refT{redbound} was deduced from \refP{prop:redbound} (together with \refC{cor:bb}) \refT{redboundtwo} may be deduced from the following proposition (\refP{prop:redboundtwo}).  We recall from \refS{sec:key} the definition of $b:(-1,\infty) \to [0,\infty)$,
\[ 
b(\eps)= (1+\eps)\log(1+\eps) - \eps \qquad \eps\in (-1,\infty).
\]

\begin{prop}\label{prop:redboundtwo} Fix $L \ge 20$. 
For any pattern $({\bf a},{\bf e})$, there exists $S=S(({\bf a},{\bf e}),L) \subset V(\Gamma)$ such that the following properties hold.
\begin{align*}
p(({\bf a},{\bf e})_{S})	& \geq p({\bf a},{\bf e}) - 150 L\sqrt{d}\qquad \text{and}\\
\sum_{(i',w')\in N_{\Gamma}(i,w)\cap S} \frac{a_{i,w}a_{i',w'}}{n} b(\eps_{i,w,i',w'})& \ge  \frac{L}{10} a_{i,w}\log\left(\frac{en}{a_{i,w}}\right) \qquad \text{for all} \, (i,w)\in S\, .\end{align*}
\end{prop}
\refT{redboundtwo} follows from \refP{prop:redboundtwo} in exactly the same way that \refT{redbound} follows from \refP{prop:redbound}.  Rather than repeat this proof we refer the reader to the proof of \refT{redbound} given in \refS{sec:key}.

We recall that our proof of \refP{prop:redbound} divided into two cases depending on whether $p_{\LD}\pat \ge p\pat/2$  or $p_{\SD}\pat \ge p\pat/2$.  These two cases were dealt with separately by \refP{prop:rbld} and \refP{prop:rbsd} respectively.  However, in \refP{prop:redboundtwo} we seek a lower bound on $p(\pat_S)$ rather than on $\tilde{p}(\pat_S)$, and for this reason we can not treat the large and small deviations cases separately. In other words, we must work with the graph $\Gamma$ rather than exclusively focussing on one of the subgraphs $\Gamma_{\LD},\Gamma_{\SD}$.  That having been said, our proof of \refP{prop:redboundtwo} resembles the proofs of Propositions~\ref{prop:rbld} and \ref{prop:rbsd} in almost all other respects.

We begin by establishing the following sufficient condition for the second bound of \refP{prop:redboundtwo} to hold.  For the remainder of the section fix a pattern $\pat$ and a constant $L \ge 20$. In what follows we denote by
\[
p(({\bf a},{\bf e});(i,w)) = \left| \sum_{(i',w')\in N_{\Gamma}(i,w)} \frac{w w' a_{i,w} a_{i',w'}}{n} \, \eps_{i,w,i',w'}\right|
\]
the `local' potency at $(i,w)$, and we recall the large and small deviations variants defined in \refS{sec:key}:
\[
p_{\LD}(({\bf a},{\bf e});(i,w)) = \left|\sum_{(i',w')\in N_{\Gamma_{\LD}}(i,w)} \frac{w w' a_{i,w} a_{i',w'}}{n} \, \eps_{i,w,i',w'}\right|
\]
and  
\[
p_{\SD}(({\bf a},{\bf e});(i,w)) = \left|\sum_{(i',w')\in N_{\Gamma_{\SD}}(i,w)} \frac{w w' a_{i,w} a_{i',w'}}{n} \, \eps_{i,w,i',w'}\right|\, .
\]
We say that a set $U \subset V(\Gamma)$ satisfies (G1), (G2), (G3), and (G4), respectively, if 
\begin{itemize}
\item[(G1)] $p(({\bf a},{\bf e})_U;(i,w)) \ge 2L a_{i,w}w^2 d^{1/2}$,
\item[(G2)] $p(({\bf a},{\bf e})_U;(i,w)) \ge 2L \widehat{N}_{i,w} /d^{1/2}$, 
\item[(G3)] $p(({\bf a},{\bf e})_U;(i,w)) \ge 2L N_i a_{i,w}/(nd^{1/2})$, or 
\item[(G4)] $p(({\bf a},{\bf e})_U;(i,w)) \ge 2L N_i m_{i,w}/d^{1/2} \, .$
\end{itemize}
for all $(i,w)\in U$, where $N_{i}=N_{i}\pat$, $\widehat{N}_{i,w}=\widehat{N}_{i,w}\pat$ and $m_{i,w}=m_{i,w}\pat$ are as defined in \refS{sec:key}.

\begin{lem}\label{lem:suff} If $U$ satisfies (G1),(G2),(G3), and (G4) then 
\[
\sum_{(i',w')\in N_{\Gamma}(i,w)\cap U} \frac{a_{i,w}a_{i',w'}}{n} b(\eps_{i,w,i',w'}) \ge  \frac{L}{10} a_{i,w}\log\left(\frac{en}{a_{i,w}}\right) \, .
\]
\end{lem}

\begin{proof} Since $E(\Gamma)=E(\Gamma_{\LD})\cup E(\Gamma_{\SD})$ it follows from the triangle inequality that $p(({\bf a},{\bf e})_S;(i,w))\le p_{\LD}(({\bf a},{\bf e})_U;(i,w)) + p_{\SD}(({\bf a},{\bf e})_U;(i,w))$, and so
\[
\max\{p_{\LD}(({\bf a},{\bf e})_U;(i,w)),p_{\SD}(({\bf a},{\bf e})_U;(i,w))\}\ge \frac{p(({\bf a},{\bf e})_U;(i,w))}{2}\, .
\]
First suppose that $p_{\LD}(({\bf a},{\bf e})_U;(i,w)) \ge p(\pat_U;(i,w))/2$. 
In this case, (G1) and (G2) imply that $U$ satisfies conditions (LD1) and (LD2), and so, by an application of \refL{lem:suffLD} we then have 
\[
\sum_{(i',w')\in N_{\Gamma_{\LD}}(i,w)\cap U} \frac{a_{i,w}a_{i',w'}}{n}
\, \cdot \, \pran{1+\frac{\eps_{i,w,i',w'}}{2}} \log(1+\eps_{i,w,i',w'}) \ge  \frac{L}{4} a_{i,w}\log\left(\frac{en}{a_{i,w}}\right)\, .
\]
The lemma now follows since (by \eqref{eq:b}) $b(\eps_{i,w,i',w'}) \ge (1+\eps_{i,w,i',w'}/2)\log(1+\eps_{i,w,i',w'})$.


By \refL{lem:suffLD} we then have
\[
\sum_{(i',w')\in N_{\Gamma_{\LD}}(i,w)\cap U} \frac{a_{i,w}a_{i',w'}}{n}
\, \cdot \, \pran{1+\frac{\eps_{i,w,i',w'}}{2}} \log(1+\eps_{i,w,i',w'}) \ge  \frac{L}{2} a_{i,w}\log\left(\frac{en}{a_{i,w}}\right)\, ,
\]
(note that the $L$ of conditions (LD1) and (LD2) became $2L$ in conditions (G1) and (G2)), which proves the result in this case.

Otherwise, we must have that $p_{\SD}(({\bf a},{\bf e})_U;(i,w)) \ge p(\pat_U;(i,w))/2$ and in this case the result follows similarly from \eqref{eq:b} and \refL{lem:suffSD}.
\end{proof}

Now let $S\subset V(\Gamma)$ be the maximal subset of $V(\Gamma)$ satisfying all of (G1),(G2),(G3), and (G4). The proof of \refP{prop:redboundtwo} is then completed by the following lemma.

\begin{lem}\label{lem:Stwo} Let $S$ be the subset defined above.  Then
\[
p(\pat_S)\ge p\pat - 150L\sqrt{d}\, .
\]
\end{lem}

\begin{proof} This is proved exactly as Lemmas~\ref{lem:potboundLD} and~\ref{lem:potboundSD} were proved; we skip the details. \qedhere \end{proof}

We now have all the necessary preliminaries in place and we may turn to our proof of \refT{explain}.

\begin{proof}[Proof of \refT{explain}] We begin by noting that the star graph $F$ consisting of a single vertex of degree $d$ attached to $d$ vertices of degree one is always a subgraph of $G$ and has $\lambda(F)=\sqrt{d}$, which proves the result when $\lambda^*(G) \leq 1189248 \sqrt{d}$. 

We may now suppose that $\lambda^*(G) = M\sqrt{d}$ for some $M \geq 1189248$.  It follows by \refP{findwit} that $G$ contains a pattern $\pat$ of potency at least $K\sqrt{d}$, where $K=K(M)=(M/192) -3\ge 6191$. By \refT{redboundtwo}, applied with $L=L(M)=K/151\ge 41$, we have that $G$ contains a reduced pattern $\pat$ with $p\pat \ge L\sqrt{d}$ for which
\begin{equation}\label{eq:probbound}
\p{({\bf a},{\bf e})~\mbox{can be found in G}}  \ll \prod_{(i,w) \in S} a_{i,w}^{d/4}\pran{\prod_{(i,w) \in V(\Gamma)} {n \choose \alpha \wedge \lfloor n/2 \rfloor}}^{-1}. 
\end{equation}
where $\alpha=\sum_{(i,w) \in V(\Gamma)} a_{i,w}$.

Let $\{A_{i,w}\}_{(i,w) \in V(\Gamma)}$ be any witness for $\pat_S$, and write 
$G'=G[\{A_{i,w}\}_{(i,w) \in V(\Gamma)}]$. By \refP{witfind} we have $\lambda(G') \geq (2L-40)\sqrt{d}-\frac{\alpha\sqrt{10}}{n}$, 
so for $n \ge \sqrt{10}hd$, if it happens that $\alpha \leq hd$ then $\lambda(G') \geq L\sqrt{d} \ge \lambda^{*}(G)/1189248$, as required. 

To complete the proof we must bound by $n^{-hd}$ the probability that any reduced pattern $\pat$ of potency at least $L\sqrt{d}$ with $\alpha>hd$ may be found in $G$.  
We remind the reader that for reduced patterns $\pat$ we have the inequality (\ref{eq:probbound}). 
We split the proof of the required bound into two steps. First, recall the event $E_1$ from the proof of \refT{main}, which was the event 
that any reduced pattern $\pat$ with $p\pat \ge L\sqrt{d}$ and for which $a_{i,w}\ge 4hd\log_2 d$ for some $(i,w) \in V(\Gamma)$, can be found in $G$. In proving \refT{main} we showed that $\p{E_1}\le n^{-hd}/2$.

Second, let $E_3=E_3(M)$ be the event that 
$G$ contains a reduced pattern $\pat$ with $p\pat \ge L\sqrt{d}$, with $\alpha=\sum_{(i,w) \in S} a_{i,w} > hd$ and with $a_{i,w} \leq 4hd\log_2 d$ for all $(i,w) \in V(\Gamma)$. We complete the proof by proving that $\p{E_3} \leq n^{-hd}/2$ for all $n$ sufficiently large. 
As in the proof of \refT{main}, by a union bound we obtain that
\[
\p{E_3(M)} \leq \log_2(nh) (4hd\log_2 d)^{3hd \log_2 d}  {n \choose \alpha \wedge \lfloor n/2 \rfloor}^{-1}.
\]
Since $hd+1 \leq \alpha= \sum_{(i,w) \in S} a_{i,w}< (4hd \log_2 d)(h \log_2 d)$, the following bound holds for all $n \geq 2\alpha$:
\[
\p{E_3(M)} \leq \log_2(nh) (4hd\log_2 d)^{3hd \log_2 d} \frac{(2(hd+1))^{hd+1}}{n^{hd+1}},
\]
which is less than $n^{-hd}/2$ for $n$ large enough. 
\end{proof}

\section{The remaining proofs}\label{sec:zbound}
Before proving \refP{zbound} we establish the following, intermediate result, 
which contains our key convexity argument. For this step, it is notationally convenient to consider vectors $x \in \R^{V(G)}$ with $\|x\|_2^2 = O(nh)$ rather than $O(1)$. 
To this end, 
define 
\begin{align*}
D & = \{0 \} \cup \{2^i: i \in \N\} \cup \{-2^i: i \in \N\} \\
\hat{D} & = \{0 \} \cup \{2^i: i \in \N\},
\end{align*}
and let 
\begin{align*}
X^* & = \{x \in \R^{V(G)}:\|x\|_2^2 \le nh\}\, , \\
Y  & = \{y \in \R^{V(G)}:\|y\|_2^2 \leq 5nh,\forall~v \in V(G),y_v \in D\}\, , \nonumber\\
Y^+& = \{y \in \R^{V(G)}:\|y\|_2^2 \leq 10nh,\forall~v \in V(G),y_v \in \hat{D}\}\, . 
\nonumber
\end{align*}
Note that by Fact~\ref{uncount}, $nh\lambda^*(G) = \sup_{x \in X^*} |\ipo{x}_N|$. 
\begin{prop}\label{dyprop}
$nh \lambda^*(G) \leq 12 \sup_{y \in Y^+} |\ipo{y}_N|.$
\end{prop}
\begin{proof}[Proof of \refP{dyprop}] 
The proof is rather straightforward; it consists of first an averaging argument, and second a polarization argument. 
The polarization argument becomes a little delicate only because we are trying to maintain the property that the 
entries of all vectors remain in $D$. 

For $r \in \mathbb{R}$ let $\sigma(r)=r/|r|$ if $r \neq 0$, and $\sigma(r)=0$ if $r=0$. We call $\sigma(r)$ the sign of $r$. 
Also, let $\ell(r)$ be the greatest element of $D$ which is less than or equal to $r$, and let $u(r)$ be the least element of 
$D$ which is greater than or equal to $r$. 

For $x \in \R^{V(G)}$, write 
\[S_x = \{ y \in \R^{V(G)}: \forall v \in V(G), \sigma(y_v)=\sigma(x), y_v \in \{\ell(x_v),u(x_v)\}\}.
\] 
Now fix any $x \in X^*$ and let $\by \in Y$ be randomly chosen as follows. Independently for each $v \in V(G)$: 
\begin{itemize}
\item if $0 \leq |x_v| < 1$ let $\by_i=\sigma(x_i)$ with probability $|x_i|$ and $\by_i=0$ otherwise.
\item if $2^j \leq |x_v| < 2^{j+1}$ for some $j \geq 1$, then let $\by_v = \sigma_v 2^{j+1}$ with probability $(|x_v|-2^j)/2^{j+1}$ and 
let $\by_v = \sigma_v 2^j$, otherwise. 
\end{itemize}
By definition, all entries of $\by$ have values in $\{0,\pm 1,\pm 2,\pm 4,\ldots\}$. Furthermore, 
for each $v$ with $|x_v| \geq 1$ we have $\by_v^2 \leq (2x_v)^2$, and for $i$ with $|x_v| < 1$ we have $\by_v^2 \leq 1$. 
It follows that 
\[
\sum_{v \in V(G)} \by_v^2 \leq \sum_{v \in V(G)} (2x_v)^2+1 = 5nh,
\]
so $\by$ is a random vector in $Y \cap S_x$. Since, for each $v \in V(G)$, $\E{\by_v}=x_v$ and the coordinates of $\by$ 
are chosen independently, we then have
\[
x = \sum_{y \in Y \cap S_x} y \p{\by=y},
\]
so 
\[
\ipo{x}_N = \sum_{y,z \in Y\cap S_x} \ip{y}{z}_N \p{\by=y}\p{\by=z}.
\]
In other words, $\ipo{x}_N$ is a convex combination of elements of $\{\ip{y}{z}_N:y,z \in Y\}$. 
Choosing $x \in X$ such that $|\ipo{x}_N|=nh\lambda^*(G)$, we then have
\begin{equation}\label{dyprop1}
nh \lambda^*(G) = |\ipo{x}{x}_N| \leq \sup_{y,z \in Y \cap S_x} |\ip{y}{z}_N|.
\end{equation}
Note that for all $y,z \in S_x$ and all $v \in V(G)$, 
either $y_v=z_v$ or else $\{y_v,z_v\}=\{\ell(x_v),u(x_v)\}$. In particular, $y_v$ and $z_v$ are either both non-negative 
or both non-positive. 

Choose $y,z \in Y \cap S_x$ for which the supremum in \refeq{dyprop1} is achieved, and write 
$y=y_+-y_-$, $z=z_+-z_-$, where e.g.~$y_+$ is the vector obtained from $y$ by replacing 
all negative entries of $y$ by zeros. 
Then for all $v \in V(G)$, $\sigma(y^+_v)=\sigma(z^+_v)$ and $\sigma(y^-_v)=\sigma(z^-_v)$ 
We then have 
\[
nh\lambda^*(G) \leq |\ip{y}{z}_N| = |\ip{y_+}{z_+}_N-\ip{y_+}{z_-}_N - \ip{y_-}{z_+}_N +\ip{y_-}{z_-}_N|,
\]
so either one of $|\ip{y_+}{z_-}_N|$,$|\ip{y_-}{z_+}_N|$ is at least $nh\lambda^*(G)/6$ or else 
one of $|\ip{y_+}{z_+}_N|$, $|\ip{y_-}{z_-}_N|$ is at least $nh\lambda^*(G)/3$. 

First suppose $|\ip{y_+}{z_-}_N| \geq nh\lambda^*(G)/6$. Since 
\begin{equation}\label{polarize}
\ipo{y_++z_-}_N = 2 \ip{y_+}{z_-}_N + \ipo{y_+}_N + \ipo{z_-}_N,
\end{equation}
it follows that either $\ipo{y_++z_-}_N| \geq nh\lambda^*(G)/9$ or else 
one of $|\ipo{y_+}_N|$ or $|\ipo{z_-}_N|$ is at least $nh \lambda^*(G)/9$. 
Also, since $y,z \in S_x$, the non-zero coordinates of $y_+$ correspond to zeros in $z_-$, we have $y_++z_- \in Y^+$. 
Since $\|y_+\|_2^2,\|z_-\|_2^2$, and $\|y_++z_-\|_2^2$ are all at most $\|y\|_2^2+\|z\|_2^2 \leq 10nh$, 
in this case we have proved the proposition. This argument 
also handles the case $|\ip{y_-^}{z_+}_N| \geq nh\lambda^*(G)/6$, by symmetry. 

The other case is that $|\ip{y_+}{z_+}_N| \geq nh\lambda^*(G)/3$ 
(the proof in the case that $|\ip{y_-}{z_-}_N| \geq nh\lambda^*(G)/3$ is symmetric). 
Since 
\[
\ipo{y_+-z_+}_N= \ipo{y_+}_N + \ipo{z_+}_N -2\ip{y_+}{z_+}_N,
\]
it follows that either one of $\ipo{y_+}_N, \ipo{z_+}_N$ is at least $nh\lambda^*(G)/12$, 
or else 
\[
\ipo{y_+-z_+}_N \geq nh\lambda^*(G)/2.
\] 
In the former case we have proved the proposition. In the latter case, note that 
since for all $v \in V(G)$, either $y_v = z_v$ or else $\{y_v,z_v\}=\{\ell(x_v),u(x_v)\}$, 
we have that all all entries of $y_+-z_+$ are in $D$, and 
furthermore, for each $v$, either $|y_{v,+}-z_{v,+}| = 0$ 
or $|y_{v,+}-z_{v,+}|=1$ or $|y_{v,+}-z_{v,+}| = |y_{v,+}| \wedge |z_{v,+}|$. It follows that $\|y_+-z_+\|^2 \leq 6nh$. 
Now write $w=y_+-z_+$, and then separate $w$ into its positive and negative 
parts: $w=w_+-w_-$. 
We then have 
\[
nh\lambda^*(G)/2 \leq |\ipo{w}_N| = | \ipo{w_+}_N- 2\ip{w_+}{w_-}_N + \ipo{w_-}_N|. 
\]
But we also have 
\[
\ipo{w_++w_-}_N= 2\ip{w_+}{w_-}_N + \ipo{w_+}_N+ \ipo{w_-}_N,
\]
and it follows that either one of $|\ipo{w_+}_N|$ or $|\ipo{w_-}_N|$ is at least $nh\lambda^*(G)/10$, 
or else $|\ipo{w_++w_-}_N| \geq nh\lambda^*(G)/10$. 
Since all of $\|w_+\|_2^2,\|w_-\|_2^2$ and $\|w_++w_-\|_2^2$ are at most $\|w\|_2^2 \leq 6nh$, 
they are all in $Y^+$ and so the proof is complete. 
\end{proof}

\begin{proof}[Proof of \refL{ipoe}]
For this proof write $N=(n_{vw})_{v,w \in V(G)}$. 
For a given vertex $v \in V(G)$, there are precisely $d$ vertices $w \in V(G)$ 
with $n_{vw}=1-1/n$, and $(n-1)d$ vertices $w \in V(G)$ with 
$n_{vw}=-1/n$ (the remaining entries in row $v$ of $N$ are all zero). Thus, for any $v \in V(G)$, 
\begin{align*}
|y_v \sum_{w \in V(G)} n_{vw} y_w \I{y_v\geq \sqrt{d}y_w}| & \leq d(1-1/n) \cdot \frac{y_v^2}{\sqrt{d}}+ (n-1)d(1/n) \frac{y_v^2}{\sqrt{d}} \\
	& < 2 y_v^2 \sqrt{d}.
\end{align*}
Writing $A= \{(x,y) \in \R^2:x \geq y\sqrt{d}\}$, it follows that
\begin{align*}
|\ipo{y}_{N,\R^2 \setminus E^*}| = 2|\ipo{y}_{N,A}|
 \leq 2 \sum_{v \in V(G)} 2y_v^2 \sqrt{d} 
		= 4\sqrt{d} \|y\|_2^2.
\end{align*}
\end{proof}

\begin{proof}[Proof of Proposition~\ref{zbound}]
Choose $z \in Y^+$ for which $|\ipo{z}_N| = |\sup_{y \in Y^+} \ipo{y}_N| \geq nh\lambda^*(G)/12$, possible by Proposition \ref{dyprop}. 
Since $\|z\|_2^2 \le 10nh$, 
by \refL{ipoe} we have 
\[
|\ipo{z}_{N,E^*}| \geq \ipo{z}_N - 40nh\sqrt{d}. 
\]
Next, for each $m \in \N_0$ let $I_m = E^* \cap [d^{m/2},d^{(m+2)/2})\times[d^{m/2},d^{(m+2)/2}) \subset \R^2$, 
and note that $E^* = \bigcup_{m \in \N} I_m$, so by the triangle inequality 
\[
|\ipo{z}_{N,E^*}| \leq \sum_{m =0}^{\infty} |\ipo{z}_{N,I_m}|. 
\]
Set 

\[
p_m = \sum_{\{v \in V(G): z_v \in [d^{m/2},d^{(m+2)/2})\}} \frac{z_v^2}{\|z\|_2^2}\, , 
\quad \alpha_m = \frac{|\ipo{z}_{N,I_m}|}{(|\ipo{z}_{N,E^*}|p_m)}\, ,
\]
so that 
\[
|\ipo{z}_{N,I_m}| = \alpha_m p_m |\ipo{z}_{N,E^*}|. 
\]
Since no point in $\R^2$ lies in more than two of the $I_m$, we have 
$\sum_{m\in \N} p_m \leq 2$, and so there must be $m^* \in \N$ for which $\alpha_{m^*} \geq 1/2$. 
Letting $z^*$ the the vector whose entry in position $v$ is $z_v\I{z_v \in I_{m^*}}$, 
we then have that 
\[
|\ipo{z^*}_{N,E}| = |\ipo{z}_{N,I_{m^*}}| \geq p_{m^*} |\ipo{z}_{N,E^*}|/2, 
\]
and furthermore, $\|z^*\|_2^2 = p_m^* \|z\|_2^2$. 
Let $y$ be the vector obtained from $z^*$ by multiplying all entries of $z^*$. 
By $2^{\lfloor -(\log_2 p_m)/2 \rfloor}$, we have 
\[
|\ipo{y}_{N,E}| \geq \frac{|\ipo{z}_{N,E^*}|}{8} \geq \frac{|\ipo{z}_N|}{8} - 5nh\sqrt{d} \geq \frac{nh\lambda^*(G)}{96}-5nh\sqrt{d}, 
\]
and $\|y\|_2^2 \leq \|z\|_2^2 \leq 10nh$.  
Finally, recall the definition of $Z$ from (\ref{zdef}). Letting $x$ be the vector obtained 
from $y$ by dividing all entries by $(nh)^{1/2}$, we obtain a vector $x \in Z$ with 
$|\ipo{x}_{N,E}| \ge \lambda^*(G)/96- 5\sqrt{d}$, which completes the proof. 
\end{proof}

We now give the promised proofs of \refP{protest} and Lemma~\ref{patcount}.

\begin{proof}[Proof of \refP{protest}]
Let $a \leq n-hn^{1/2}$ be the number of vertices in $G'$, and let $x' \in \R^{V(G')}$ be an eigenvector of $G'$ with eigenvalue $\lambda$, chosen so that $\|x'\|_2^2 = 1$. For each $i \in [h]$, let $t_i = \sum_{v \in V_i \cap V(G')} x_v$. 
We define a vector $y \in \R^{V(G)}$, for all $i \in [h]$ and each $v \in V_i \setminus V(G')$, setting $y_v = - t_i/|V_i \setminus V(G')|$, and taking all other $y_v$ equal to zero. Taking $x=x'+y$, it is immediate that $x$ is balanced. 
Also, for $v \in V(G)\setminus V(G')$, we have $|x_v|=|y_v| \leq 1/(n-a)$, and it follows that $\|x\|_2^2 \leq 1+nh/(n-a)^2 \leq 1+1/\lambda$, since $(n-a) \geq hn^{1/2} > (nd^2)^{1/2} \geq (nd\lambda)^{1/2}$. Now, 
\begin{align*}
\ipo{x}_N 	& = \ipo{x}_M \\
			& = \ipo{x'}_M + \ipo{y}_M + 2\ip{x'}{y}_M \\
			& = \lambda + \ipo{y}_M + 2\ip{x'}{y}_M.
\end{align*}
Since all entries of $y$ have modulus at most $(n-a)^{-1}$, it follows that $|\ipo{y}_M| \leq |E(G)|/(n-a)^2 = nhd/(2(n-a)^2) \leq n^{1/2}d/(2(n-a))$. 
Similarly, since $\sum_{v \in V(G')} |x_v| \leq a^{1/2}$, we have $2|\ip{x'}{y}_M| \leq 2a^{1/2}d/(n-a)$. 
It follows that 
\[
|\ipo{x}_N| \geq \lambda - \frac{d(n^{1/2}+4a^{1/2})}{2(n-a)} \geq \lambda - \frac{5dn^{1/2}}{2(n-a)} > \lambda - \frac{5}{2},
\]
and recalling that $\|x\|_2^2 \leq 1+1/\lambda < \lambda/(\lambda-1)$ completes the proof. 
\end{proof}

\begin{proof}[Proof of Lemma~\ref{patcount}]
We may specify a pattern by first choosing $w_0$, then choosing $a_{i,w}$ for each $(i,w)$ with $i \in [h]$ 
and $w \in \dog \cap [w_0,w_0d]$, and finally choosing $f_{i,w,i',w'}$ for each pair $(i,w),(i',w') \in V(\Gamma)$ 
with $i\simh i'$. 

If the pattern is to be non-empty then there is some $(i,w) \in V(\Gamma)$ for 
which $\alpha_{i,w} nh/w^2$ is a positive integer. Since $\alpha_{i,w} \leq 1$ it follows that $w_0^2 \leq w^2 \leq nh$ 
and so there are at most $|\dog \cap [0,\sqrt{nh}]| \leq \log_2(nh)$ choices for $w_0$. 

Having chosen $w_0$, since $|\dog \cap [w_0,w_0d]| \leq \log_2 (2d)$, there are at most 
$A^{h\log (2d)}$ choices for the $a_{i,w}$. 
Finally, for each pair $(i,w),(i',w')$ with $i \simh i'$, there are at most $1+\min(a_{i,w},a_{i',w'}) \leq A$ choices 
for $f_{i,w,i',w'}$. There are at most $hd \log_2(2d)/2$ such pairs, so the 
total number of choices for the $f_{i,w,i',w'}$ is at most $A^{hd \log_2(2d)/2}$. 

Combining the bounds of the two preceding paragraphs, we obtain that the total number of patterns 
with $a_{i,w} < A$ for all $(i,w) \in V(\Gamma)$ is at most 
\[
\log_2(nh) \cdot A^{h \log_2(2d)} A^{hd \log_2(2d)/2}< \log_2(nh) \cdot A^{2h d\log_2 d}. 
\]
\end{proof}


\section{Acknowledgements} {\Small S.G. heard of this problem in a seminar of Benny Sudakov at the IPAM long program on ``Combinatorics: Methods and Applications in Mathematics and Computer Science''.  He would like to thank IPAM for providing such an excellent program.  Most of the research for this article took place while he was a postdoctoral fellow at McGill University, he would like to thank them for their support and facilities.  He is currently supported by CNPq (Proc. 500016/2010-2).  L.A-B. was supported during this research by an NSERC Discovery Grant. Both authors would like to thank two anonymous referees for many valuable comments and suggestions.} 

\bibliographystyle{plainnat}

                 

\end{document}